\documentclass[final,hidelinks,onefignum,onetabnum]{siamart250211}
\usepackage{geometry}
\usepackage{bm,mathtools}
\usepackage{algorithm}
\usepackage{booktabs}
\usepackage{mathrsfs}

\usepackage{lipsum}
\usepackage{amsfonts}
\usepackage{graphicx}
\usepackage{epstopdf}
\usepackage{algorithmic}
\ifpdf
  \DeclareGraphicsExtensions{.eps,.pdf,.png,.jpg}
\else
  \DeclareGraphicsExtensions{.eps}
\fi

\newsiamremark{remark}{Remark}
\newsiamremark{hypothesis}{Hypothesis}
\crefname{hypothesis}{Hypothesis}{Hypotheses}
\newsiamthm{claim}{Claim}
\newsiamremark{fact}{Fact}
\crefname{fact}{Fact}{Facts}

\headers{Unconditional Optimal Error Estimates and Energy Stability}{Yongyong Cai and Xingwei Yang }

\title{Unconditional Optimal Error Estimates and Energy Stability for a Linearly Implicit Mass-Lumped Projection Finite Element Method for the Harmonic Map Flow}

\author{Yongyong Cai\thanks{Laboratory of Mathematics and Complex Systems and School of Mathematical Sciences, Beijing
Normal University, Beijing 100875, China (yongyong.cai@bnu.edu.cn)}
\and Xingwei Yang\thanks{Corresponding author. Laboratory of Mathematics and Complex Systems and School of Mathematical Sciences, Beijing
Normal University, Beijing 100875, China (\email{xingwei.yang@mail.bnu.edu.cn}).}}

\usepackage{amsopn}

\begin{document}
\newcommand{\norm}[1]{\left\Vert#1\right\Vert}
\newcommand{\curl}{\nabla\times}
\newcommand{\grad}{\nabla}
\newcommand{\Qed}{\hfill\qedsymbol}
\renewcommand{\div}{\nabla\cdot}
\newcommand{\R}{\mathbb{R}}
\newcommand{\m}{\mathbf{m}}
\newcommand{\ppsi}{\bm{\psi}}
\renewcommand{\H}{\mathbf{H}}
\newcommand{\B}{\mathbf{B}}
\renewcommand{\v}{\mathbf{v}}
\newcommand{\w}{\mathbf{w}}
\renewcommand{\a}{\mathbf{a}}
\renewcommand{\b}{\mathbf{b}}
\renewcommand{\u}{\mathbf{u}}
\newcommand{\s}{\mathbf{s}}
\newcommand{\J}{\mathbf{J}}
\newcommand{\F}{\mathbf{F}}
\newcommand{\E}{\mathbf{E}}
\newcommand{\Ep}{\mathcal{E}}
\renewcommand{\S}{\mathbf{S}}
\newcommand{\Q}{\mathbf{Q}}
\newcommand{\eff}{\rm{eff}}
\newcommand{\n}{\mathbf{n}}
\newcommand{\f}{\mathbf{f}}
\newcommand{\g}{\mathbf{g}}
\newcommand{\M}{\mathbf{M}}
\newcommand{\x}{\mathbf{x}}
\newcommand{\z}{\mathbf{z}}
\renewcommand{\d}{\mathrm{d}}
\renewcommand{\hat}{\widehat}
\newcommand{\e}{\mathbf{e}}
\renewcommand{\l}{\Big(}
\renewcommand{\r}{\Big)}
\renewcommand{\tilde}[1]{\widetilde{#1}}
\allowdisplaybreaks[4]
\maketitle

\begin{abstract}
We propose and analyze a linearly implicit mass-lumped finite element
method for the heat flow of harmonic maps into the unit sphere. The
method consists of a linear predictor followed by a nodal projection
and therefore preserves the unit-length constraint exactly at all finite
element nodes. The predictor is derived from a cross-product
reformulation of the equation and is shown to be equivalent to a
mass-lumped discretization of the original formulation with a correction
term enforcing nodal orthogonality, as well as to a tangent
plane scheme.
A key ingredient is the consistent use of the discrete inner product in
both the mass and stiffness terms. This yields a nodal orthogonality
relation implying that the auxiliary solution lies on or outside the unit sphere
at every node. Consequently, the projection is well defined and the
projected error satisfies a contraction property in the discrete
\(L^2\)-norm. On Cartesian rectangular and cuboidal tensor-product
meshes, the nodal projection is also nonexpansive in a discrete
Dirichlet energy, which gives an unconditional discrete energy
dissipation law.
For sufficiently smooth solutions, we prove optimal error estimates
without any coupling condition between the time step and the mesh size:
the method converges with order \(O(\Delta t+h^2)\) in
\(\ell^\infty(0,T;L^2)\) and order \(O(\Delta t+h)\) in
\(\ell^2(0,T;H^1)\). The proof combines the projected-error contraction,
quadrature consistency estimates, edge-based cancellation identities,
and a bootstrap argument for controlling nonlinear terms. Numerical
experiments confirm the predicted convergence rates and the discrete
energy decay.
\end{abstract}


\section{Introduction}
For a bounded Lipschitz domain $\Omega\subset \mathbb{R}^d$ with $d \in \{1, 2, 3\}$ and given $T>0$, the heat flow of harmonic maps reads as 
\begin{align}
  \partial_t \m &= \Delta \m+|\grad \m|^2\m\quad &&\text{ in }\Omega\times (0,T],\label{a7}\\
  \partial_{\mathbf{n}}\m &= \mathbf{0}\quad &&\text{ on }\partial\Omega\times(0,T],\\
  \m& = \m^0\quad &&\text{ in }\Omega\times \{0\},\label{c4}
\end{align}
where the initial data $\m^0\in \H^1(\Omega)$ satisfies $|\m^0(\x)| = 1$ for almost every $\x\in \Omega$. 
The symbol $|\cdot|$ denotes the Euclidean norm for a vector and the Frobenius norm for a matrix, $\grad$ is the gradient operator, $\Delta$ is the Laplacian operator, $\mathbf{n}$ is the outward unit normal vector on the boundary $\partial\Omega$. Moreover, $\partial_t\m$ and $\partial_{\mathbf{n}}\m$ denote the time derivative and the normal derivative of $\m$, respectively. Denoting by $\mathbb{S}^2$ the unit sphere in $\mathbb{R}^3$, the solution to \eqref{a7}--\eqref{c4} intrinsically belongs to $\mathbb{S}^2$, i.e., $\m(\x, t)$ satisfies the pointwise constraint
\begin{equation}\label{e4}
  |\m(\x, t)|=1\quad \text{ in }\Omega\times(0,T].
\end{equation}
It also satisfies the following energy identity for $a.e.$ $t\in (0,T]$:
\[
E(\m(\x, t))+\int_0^t \norm{\m(\x, s)\times\Delta\m(\x, s)}_{L^2}^2\ \d s = E(\m^0),\quad E(\m) \coloneqq \frac{1}{2}\int_\Omega |\grad \m|^2 \mathrm{d}\mathbf{x}.
\]
The relevant results on local existence, uniqueness, and finite-time blow-up behavior for the system \eqref{a7}--\eqref{c4} can be found in \cite{MR1705180, MR990191}.
As a fundamental equation in many models whose exact solution inherently satisfies a unit-length constraint, the heat flow of harmonic maps has natural applications in various physical scenarios, e.g., the Ericksen--Leslie model for nematic liquid crystal flow \cite{MR2399392}, the Landau--Lifshitz equation for magnetization dynamics \cite{landau1935theory, gilbert1955lagrangian}, color image denoising \cite{MR1974176}, and mean curvature flow \cite{MR4026373}.

For the numerical approximation of the evolution problem \eqref{a7}--\eqref{c4} and the related Landau-Lifshitz equation, optimal error estimates of finite element methods (FEMs) have been studied recently. Gao established the optimal unconditional convergence result for the Landau-Lifshitz equation in \cite{MR3273326} by a new linearization and the error splitting technique developed in \cite{MR3072763}. Akrivis et al.\ \cite{MR4232216} derived optimal-order error estimates using high-order backward difference formula (BDF) time discretizations. More recently, an optimal $H^1$ error estimate was established for a tangent plane finite element scheme applied to the heat flow of harmonic maps, subject to the condition $\Delta t\le C_mh^{\frac{1}{2}}$ for sufficiently small $C_m$ \cite{MR4727106}. Nevertheless, these schemes only preserve the unit-length constraint approximately. For the development of constraint accuracy for the heat flow of harmonic maps, we refer to \cite{MR5056183, MR3454357}.

To preserve the unit-length constraint \eqref{e4} at the discrete level, there are many works including the penalization method \cite{MR1885923}, the projection method \cite{MR1813249}, the midpoint scheme \cite{MR2257110, MR2152000}, the tangent plane projection scheme \cite{MR2210092, MR2379897}, and the Lagrange multiplier method \cite{MR3719027, MR4579735}. Due to the simplicity of its implementation, the projection approach has been used in various applications. Optimal convergence results for the backward Euler and Crank-Nicolson semi-implicit finite difference projection methods were derived under the conditions $h^2\le \Delta t\le h^{1+\epsilon_1} $ and $c_1h\le \Delta t\le c_2h$ respectively in \cite{MR4272916}, where $\epsilon_1\in (0,1)$ is a fixed constant. In \cite{MR4454924}, the optimal convergence rates for the backward Euler FEM were established under the conditions $c_1h\le \Delta t\le c_2h$ and $r\ge 2$, where $c_1$ and $c_2$ are fixed positive constants and $r$ is the polynomial degree of the finite element space. A weaker condition $\Delta t\ge \kappa h^{r+1}$ for optimal-order error estimates is derived by employing tensor-product finite elements on rectangular and cuboidal meshes \cite{MR4377027}. Extension to the Landau-Lifshitz equation with the Crank-Nicolson scheme can be found in \cite{MR5004067}. The multistep projection methods for the harmonic map heat flow into general surfaces were presented in \cite{MR4883780}, where a similar constraint $\kappa h^{r+1}\le \Delta t^k\le \Delta t_0$ was imposed for the $k$-step BDF time discretization.
In \cite{MR4966454}, the optimal $H^1$ error estimate for the normalized tangent plane method was established under the time step condition $O(h^{\frac{7}{4}})\le \Delta t\le O(h)$. Some semi-discrete convergence results can also be found in \cite{MR1681053,MR5035027}. 

To the best of our knowledge, an optimal-order error estimate without
any coupling condition between \(\Delta t\) and \(h\) has not been
available for fully discrete finite element schemes that
preserve the unit-length constraint at all nodes. The present paper
closes this gap for a mass-lumped FEM on Cartesian
rectangular and cuboidal tensor-product meshes. The analytical mechanism developed here may be useful for related geometric-flow systems, such as $p$-harmonic maps, Landau--Lifshitz-type
equations, and certain nematic liquid crystal models. A rigorous extension to these systems, however, requires additional arguments and is left for future work. The key observation behind the proposed scheme is as follows.

For a typical projection method, the time derivative gives rise to a
term of the form
\[
\left(\frac{\tilde \m_h^{n+1}-\m_h^n}{\Delta t},\v_h\right),
\qquad
\m_h^{n+1}=I_h\left(\frac{\tilde \m_h^{n+1}}
{|\tilde \m_h^{n+1}|}\right),
\]
where \(I_h\) is the Lagrange interpolation operator, \(\m_h^n\) is the
projected solution from the previous time level, \(\tilde\m_h^{n+1}\) is
the current auxiliary solution, and \(\v_h\) is a test function.
The corresponding error equation contains
\[
\left(\frac{\tilde \e_h^{n+1}-\e_h^n}{\Delta t},\v_h\right),
\]
and therefore the relation between the projected error \(\e_h^{n+1} = I_h\m(t_{n+1})-\m_h^{n+1}\)
and the predictor error \(\tilde\e_h^{n+1} = I_h\m(t_{n+1})-\tilde{\m}_h^{n+1}\) becomes a central issue. In \cite{MR4268651}, the following relation
\[
\norm{\e_h^{n+1}}_h^2\le C\norm{\tilde{\e}_h^{n+1}}_h^2,
\]
is employed for the semi-renormalized scheme, where $\norm{\cdot}_h$ denotes the discrete $L^2$ norm. However, reliance on this estimate precludes the extension of the error analysis to a fully renormalized FEM unless an additional stepsize constraint $\Delta t\ge h$ is imposed. An improved relation
\begin{equation}\label{e5}
\norm{\e_h^{n+1}}_h^2\le \norm{\tilde{\e}_h^{n+1}}_h^2 + \text{higher-order terms}
\end{equation}
under a weaker condition $\Delta t\ge \kappa h^{r+1}$ is established in \cite{MR4377027}. The relation \eqref{e5} depends on the property that $|\m_h^{n+1}(\z)|=1$ for all finite element nodes $\z$, which naturally motivates using the discrete inner product for the discrete time derivative
\[
\left(\frac{\tilde{\m}_h^{n+1} - \m_h^{n}}{\Delta t}, \v_h\right)_h.
\]
Considering the analogous form of the error equation again, the following
inherent property of the numerical scheme is expected to eliminate the
time-step restriction arising from the relation between the errors:
\begin{align}
\norm{ \e_h^{n+1}}_h\le \norm{ \tilde\e_h^{n+1}}_h.\label{a9}
\end{align}
It is worth noting that the relation \eqref{a9} has been derived for the tangent plane projection scheme in \cite{MR4966454}; however, a restrictive  condition $O(h^{\frac{7}{4}})\le \Delta t\le O(h)$ is required to ensure that the sum of projection errors remains controllable in the standard $L^2$-norm. Inspired by these observations, the construction of a fully discrete scheme that intrinsically satisfies property \eqref{a9} relies on the fact that the auxiliary solution lies on or outside the unit sphere at the finite element nodes, i.e.,
\begin{align}\label{c7}
|\tilde{\m}_h^{n+1}(\z)|\ge |\m_h^{n}(\z)|=1\quad \forall \z\in \mathcal{Q},
\end{align}
where $\mathcal{Q}$ is the set of all finite element nodes. The condition \eqref{c7} is a direct consequence of the orthogonality condition
\begin{align}\label{c6}
&\frac{\tilde{\m}_h^{n+1}(\z) - \m_h^{n}(\z)}{\Delta t}\cdot\m_h^n(\z) = 0, \quad \forall \z\in \mathcal{Q}.
\end{align}
Motivated by this, we use the identity
\begin{equation}\label{o6}
\mathbf{a}\times(\mathbf{b}\times \mathbf{c}) = (\mathbf{a}\cdot \mathbf{c})\mathbf{b}-(\mathbf{a}\cdot \mathbf{b})\mathbf{c}, \quad \mathbf{a}, \mathbf{b}, \mathbf{c} \in \mathbb{R}^3
\end{equation}
to rewrite \eqref{a7}--\eqref{c4} into the following cross-product form
\begin{align}
  \partial_t \m &= -\m\times(\m\times\Delta \m) \quad &&\text{ in }\Omega\times (0,T],\label{t4}\\
  \partial_{\mathbf{n}}\m &= \mathbf{0}\quad &&\text{ on }\partial\Omega\times(0,T],\\
  \m& = \m^0\quad &&\text{ in }\Omega\times \{0\},\label{t5}
\end{align}
and propose the following scheme based on the above equivalent form with $\m_h^0 = \tilde\m_h^0 = I_h\m^0$
\begin{align}
\l \frac{\tilde{\m}_h^{n+1} - \m_h^{n}}{\Delta t},\v_h \r_h=& -\l\m_h^n\times\l\m_h^n\times\tilde\Delta_h \tilde{\m}_h^{n+1}\r,\v_h\r_h,\label{a2}\\
\m_h^{n+1} =& I_h\l \frac{\tilde{\m}_h^{n+1}}{|\tilde{\m}_h^{n+1}|}\r,\label{a8}
\end{align}
where
\[
    (\tilde\Delta_h\v_h, \bar\v_h)_h = -(\grad\v_h,\grad\bar\v_h)_h, \quad \forall\v_h, \overline{\v}_h\in \S_h^1.
\]
In contrast to schemes where mass lumping is applied only to the time
derivative, we also evaluate the stiffness term by
quadrature. Since the gradient of a \(Q_1\) function is taken
elementwise, this defines a broken quadrature-based stiffness form. This choice is essential for the nodal orthogonality relation and for
the unconditional decay of the discrete Dirichlet energy. There are some studies that use quadrature to compute the stiffness term; see, e.g., \cite{MR816063, MR4431066, MR4529401}. 
Note that one can also use the standard inner product to define the present discrete Laplacian operator $\tilde\Delta_h$ (see \eqref{b4}). Although this change does not affect the key property \eqref{a9}, the unconditional discrete energy stability and optimal error analysis would be much more difficult to establish.

Given that $|\m_h^{n}(\z)| = 1$ for all $\z\in \mathcal{Q}$, choosing the test function $\v_h = \phi_{\z}\m_h^n(\z)$ in \eqref{a2}, where $\phi_{\z}$ denotes the nodal basis function at $\z$, we obtain the orthogonality condition \eqref{c6} on every finite element node.
It follows directly that condition \eqref{c7} is satisfied. The well-definedness of the projection step is thus ensured. Finally, the proof of relation \eqref{a9} for scheme \eqref{a2}--\eqref{a8} is provided in Lemma~\ref{lemma4}.

Although the difficulty arising from the projection step is overcome by the relation \eqref{a9}, the use of quadrature-based discrete inner products in the fully discrete formulation makes it necessary to control the resulting
quadrature errors. In addition to standard estimates based on the Bramble--Hilbert lemma, we establish edge-based cancellation estimates induced by the nodal unit-length and orthogonality relations, which are crucial for controlling the quadrature error and the nonlinear terms in the error equation.  Moreover, uniform \(L^\infty\)- and \(W^{1,4}\)-type
bounds for the errors and the numerical solution are obtained without
any coupling condition between \(\Delta t\) and \(h\), by combining
discrete Sobolev inequalities, norm equivalences, the comparison between
the two discrete Laplacians, and a bootstrap induction argument. These bounds are crucial for controlling the nonlinear terms in the error equation.  To guarantee that the projected solution is energy-decreasing, two aspects need to be satisfied simultaneously: i) the predictor scheme is energy-decreasing and ii) the projection step is also energy non-increasing, i.e., the projection is energy stable \cite{MR1472192,MR2177142}. For this purpose, we first establish a stability estimate for the nodal
projection on quadrilateral/hexahedral meshes by using \eqref{c7}. In the Cartesian rectangular/cuboidal case, this
estimate becomes nonexpansive in the discrete Dirichlet energy. Combining
this projection nonexpansiveness with the energy decay of the predictor
step \eqref{a2}, we obtain the discrete energy decay for the fully
projected scheme.

The rest of the paper is organized as follows. Section~\ref{sec2}
introduces the notation, the discrete setting, and equivalent
formulations of the proposed scheme. Section~\ref{sec3} states the main
stability and convergence results. Section~\ref{sec4} is devoted to the
proof of the unconditional optimal-order error estimates. Numerical
experiments are reported in Section~\ref{sec5}. Finally, concluding remarks are given in Section \ref{sec6}.
\section{Preliminaries and equivalent formulations}\label{sec2}
\subsection{Notations}
Throughout this paper, $c$, $C$ (with or without subscripts) and $\epsilon$ denote generic positive constants that are independent of the spatial and temporal discretization parameters, where $\epsilon$ can be taken arbitrarily small. We employ $L^p(\Omega)$, $H^k(\Omega)$, and $W^{k,p}(\Omega)$ for the usual Sobolev spaces, with corresponding norms $\norm{\cdot}_{L^p}, \norm{\cdot}_k$ and $\norm{\cdot}_{W^{k,p}}$. The standard notations $(\cdot, \cdot)$ and $\norm{\cdot}$ denote the $L^2(\Omega)$ inner product and norm, respectively. Vector-valued function spaces are denoted in boldface. For simplicity, the domain $\Omega$ will be omitted in the notation of function spaces hereafter. 
\subsection{The discrete settings}
The total time $T$ is partitioned into $N$ uniform time intervals $[t^{n}, t^{n+1}]$ for $ n=0, \cdots, N-1$, where $t^n = n\Delta t$ and $\Delta t = T/N$.
Let $\Omega$ be a rectangular domain with boundary $\partial\Omega$, and let $\mathscr{T}_h$ be a rectangular (or cuboidal) mesh generated by the tensor product of one-dimensional partitions in each coordinate direction, satisfying shape-regularity and quasi-uniformity. The mesh size is denoted by $h = \max_{K\in \mathscr{T}_h} h_K$, where $h_K$ is the diameter of the element $K$. 

Define the tensor-product finite element space by $\S_h^1\coloneqq\{ \v_h\in \H^1(\Omega):\v_h|_K\in Q_1(K)^3\ \forall K\in\mathscr{T}_h\}$ where $Q_1(K)$ is the space of polynomials of degree at most $1$ in each variable on $K$. Let \(\mathcal Q\) denote the set of all global finite element nodes. The mass lumping is achieved by using the $2$-point Gauss--Lobatto quadrature rule, which guarantees that the quadrature nodes coincide with the finite element nodes. This setting is closely related to the spectral element framework, where
tensor-product elements and Gauss--Lobatto quadrature are commonly used. For an element $K \in \mathscr{T}_h$, let \(\mathcal Q_K\) denote the set of its vertices and let $\omega_{K, \z}$ denote the corresponding quadrature weight at the quadrature point $\z\in \mathcal{Q}_{K}$. We denote by $I_h$ the global continuous $Q_1$ interpolation operator
associated with the global nodes, and by $I_K$ its local counterpart on
an element $K$.

We define the following discrete inner product for any $\v_h, \v'_h \in \S_h^1$:
\begin{equation}
(\v_h, \v'_h)_h = \sum_{K\in \mathscr{T}_h} \sum_{\z\in \mathcal{Q}_{K}}\omega_{K, \z}\v_h(\z)\cdot\v'_h(\z).
\end{equation} 
Since $I_h(\v_h \cdot\v'_h)(\z) = \v_h(\z) \cdot\v'_h(\z)$ for all $\z\in\mathcal{Q}_K$ and all $\v_h,\v'_h\in\S_h^1$, and since the $2$-point Gauss--Lobatto quadrature is exact for $Q_1$ functions, the discrete inner product admits the equivalent expression
\[
(\v_h, \v'_h)_h = \int_{\Omega} I_h(\v_h\cdot\v'_h)\ \d\x.
\]
Since the quadrature weights are strictly positive, we can define the following discrete $L^2$ norm on $\S_h^1$:
\[
\norm{\v_h}_h\coloneqq \norm{\v_h}_{L_h^2} = \sqrt{(\v_h,\v_h)_h}, \quad \forall \v_h\in \S_h^1.
\]
Similarly, the discrete $L^p$ norm $\norm{\cdot}_{L_h^p}$ is defined by
\begin{equation*}
  \norm{\v_h}_{L_h^p}\coloneqq\left\{
  \begin{aligned}
    &\l \sum_{K\in \mathscr{T}_h}\int_{K} I_h(|\v_h|^p)\ \d \x  \r^{\frac{1}{p}}&&\text{ for } \v_h\in \S_h^1, 1\le p<\infty,\\
    &\max_{\z\in \mathcal{Q}}|\v_h(\z)| &&\text{ for }\v_h\in \S_h^1, p=\infty. 
  \end{aligned}\right.
\end{equation*}
To define the discrete Laplacian operator \(\widetilde\Delta_h\), we
introduce a discrete inner product for elementwise
gradients:
\begin{align*}
&(\grad\v_h,\grad\v'_h)_{K,h} = \sum_{\z\in\mathcal Q_K}\omega_{K, \z}\grad(\v_h|_K)(\z)\cdot\grad(\v'_h|_K)(\z),
\end{align*}
where $\grad(\v_h|_K)$ denotes the gradient of the restriction of $\v_h$ on $K$ and $\cdot$ between two matrices is the Frobenius inner product. The discrete inner product for the stiffness term can also be equivalently expressed as
\begin{align*}
&(\grad \v_h, \grad \v'_h)_h = \sum_{K\in\mathscr T_h}(\grad\v_h,\grad\v'_h)_{K,h} = \sum_{K\in\mathscr T_h}\int_K I_K(\grad(\v_h|_K)\cdot\grad(\v'_h|_K))\ \d\x.
\end{align*}
The corresponding discrete broken gradient norm is defined as follows:
\begin{equation}\label{e6}
  \norm{\grad\v_h}_{L_h^p}\coloneqq\left\{
  \begin{aligned}
    &\l\sum_{K\in\mathscr T_h}\int_{K} I_K(|\grad(\v_h|_K)|^p)\ \d \x  \r^{\frac{1}{p}}&&\text{ for } \v_h\in \S_h^1, 1\le p<\infty,\\
    &\max_{K\in\mathscr T_h}\max_{\z\in\mathcal Q_K}|\grad(\v_h|_K)(\z)|&&\text{ for }\v_h\in \S_h^1, p=\infty.
  \end{aligned}\right.
\end{equation}
The above definitions are extended in the same way to elementwise smooth functions, which need not belong to $\S_h^1$ or be globally continuous. Moreover, the discrete $H^1$-norm can be defined as
\[
  \norm{\v_h}_{H^1_h}^2 = \norm{\v_h}_h^2+\norm{\grad\v_h}_h^2.
\]

By mapping each element to the reference rectangle/cuboid and using the equivalence of norms on the finite-dimensional tensor-product polynomial
space, one obtains the following equivalence between the standard and discrete norms:
\begin{align}
  C_1\norm{\v_h}_{L^p} &\le \norm{\v_h}_{L_h^p} \le C_2\norm{\v_h}_{L^p}, \label{b3}\\
  C_3\norm{\grad\v_h}_{L^p} &\le \norm{\grad\v_h}_{L_h^p} \le C_4\norm{\grad\v_h}_{L^p}, \label{p7}
\end{align}
for all $\v_h\in \S_h^1$ and $1\le p\le +\infty$, where the constants $C_1,\dots,C_4>0$ are independent of $h$ and \eqref{b3} can be found in \cite{MR4377027, MR4883780}.
Consequently, the standard finite element inverse inequalities and
interpolation inequalities remain valid in the corresponding discrete
norms. In particular, for \(1\le p\le\infty\),
\begin{align}
  &h\|\nabla\v_h\|_{L_h^p}
  \le C\|\v_h\|_{L_h^p},\quad\|\nabla\v_h\|_{L_h^4}\le C h^{-1-d/4}\|\v_h\|_h\label{e7}\\
  &\|\v_h\|_{L_h^4}\le C\|\v_h\|_h^{1/4}\|\v_h\|_{L_h^6}^{3/4}\le C\|\v_h\|_h^{1/4}\|\v_h\|_{H_h^1}^{3/4}.\label{t1}
\end{align}

\subsection{Elementwise edge differences}
Note that the global Cartesian tensor-product property guarantees certain relations between the FEM and the finite difference method in the present setting.
For later use, we introduce an elementwise edge-difference notation. Let
\[
  K=\prod_{j=1}^d[x_{K,j}^-,x_{K,j}^+],
  \qquad
  h_{K,j}=x_{K,j}^+-x_{K,j}^- .
\]
For $i\in\{1,\ldots,d\}$ and
$\alpha'\in\{0,1\}^{d-1}$, let
$\z_{K,i,\alpha'}^-$ and $\z_{K,i,\alpha'}^+$ be the two vertices of $K$
which differ only in the $i$-th coordinate. Define the local edge difference
\[
  \mathfrak D_{K,i}^{\alpha'}\v_h \coloneqq \frac{\v_h(\z_{K,i,\alpha'}^+)-\v_h(\z_{K,i,\alpha'}^-)}{h_{K,i}} .
\]
If $\v_h|_K\in Q_1(K)$, we have
\begin{align}\label{t2}
  \partial_i (\v_h|_K)(\z_{K,i,\alpha'}^-)
  =
  \partial_i (\v_h|_K)(\z_{K,i,\alpha'}^+)
  =
  \mathfrak D_{K,i}^{\alpha'}\v_h .
\end{align}
Since the quadrature weight $\omega_{K,\z}$ is the same for all nodes of a given element $K$, we denote it by $\omega_K$. The following identity then holds for all $\v_h\in \S_h^1$:
\begin{equation}\label{o1}
  \norm{\partial_i \v_h}_{L_h^2(K)}^2
  = 2\omega_{K} \sum_{\alpha'\in\{0,1\}^{d-1}} |\mathfrak D_{K,i}^{\alpha'}\v_h|^2.
\end{equation}
\subsection{The equivalent schemes}
The scheme \eqref{a2}--\eqref{a8} is the mass-lumped FEM discretization of the equivalent equation \eqref{t4}--\eqref{t5}.
We now show that it is also equivalent to the mass-lumped FEM discretization of the original equation \eqref{a7} combined with a correction term that guarantees property \eqref{c6}. We shall use the following edge-based cancellation identity; the proof can be found in Appendix \ref{t3}.
\begin{lemma}\label{y9}
  Assume that \(|\m_h^n(\z)|=1\) for all \(\z\in\mathcal Q\).
  For any \(\v_h\in\S_h^1\), let $
  \F_h:=\m_h^n(\m_h^n\cdot\v_h)$. Then we have
\begin{align}\label{t9}
  \l\grad \m_h^n,\grad(\F_h-I_h\F_h)\r_h
  =
  \l\m_h^n\cdot\grad\m_h^n,\grad(\m_h^n\cdot\v_h)\r_h .
\end{align}
\end{lemma}
Now, using \eqref{o6}, the definition of the discrete Laplacian operator $\tilde\Delta_h$ and Lemma \ref{y9}, \eqref{a2} can be rewritten as follows:
\begin{align}
&\l \frac{\tilde{\m}_h^{n+1} - \m_h^{n}}{\Delta t},\v_h \r_h+\l\grad \tilde{\m}_h^{n+1}, \grad\v_h\r_h
=- \l\tilde\Delta_h \tilde{\m}_h^{n+1},I_h(\m_h^n(\m_h^n\cdot\v_h))\r_h \notag\\
=&- \l\tilde\Delta_h (\tilde{\m}_h^{n+1} - \m_h^n),I_h(\m_h^n(\m_h^n\cdot\v_h))\r_h - \l\tilde\Delta_h \m_h^n,I_h(\m_h^n(\m_h^n\cdot\v_h))\r_h\notag\\
=&\l\grad (\tilde{\m}_h^{n+1} - \m_h^n),\grad I_h(\m_h^n(\m_h^n\cdot\v_h))\r_h+\l\grad \m_h^{n},\grad(\m_h^n(\m_h^n\cdot\v_h))\r_h\notag\\
&+\l\grad \m_h^{n},\grad\l I_h(\m_h^n(\m_h^n\cdot\v_h)) - \m_h^n(\m_h^n\cdot\v_h)\r\r_h\notag\\
=&\l|\grad\m_h^n|^2,\m_h^n\cdot\v_h\r_h + \l\grad (\tilde{\m}_h^{n+1} - \m_h^n),\grad I_h(\m_h^n(\m_h^n\cdot\v_h))\r_h.\notag
\end{align}
The second term on the right-hand side is the correction term. It is precisely this term that preserves the
nodal orthogonality relation \eqref{c6} in the equivalent formulation of
the original equation.
For the convenience of the subsequent analysis, we rewrite the above equation into the following form, which is easier to analyze. (One can also derive it directly from \eqref{a2}.)
\begin{align}\label{u2}
&\l \frac{\tilde{\m}_h^{n+1} - \m_h^{n}}{\Delta t},\v_h \r_h+\l\grad \tilde{\m}_h^{n+1}, \grad\v_h\r_h\notag\\
=&\l\grad (\tilde{\m}_h^{n+1} - \m_h^n),\grad\l I_h(\m_h^n(\m_h^n\cdot\v_h)) - \m_h^n(\m_h^n\cdot\v_h)\r\r_h\notag\\
&+\l\grad\m_h^n\cdot\grad \tilde{\m}_h^{n+1},\m_h^n\cdot\v_h\r_h+\l\m_h^n\cdot\grad (\tilde{\m}_h^{n+1} - \m_h^n),\grad(\m_h^n\cdot\v_h)\r_h\notag\\
=&\l\grad\m_h^n\cdot\grad \tilde{\m}_h^{n+1},\m_h^n\cdot\v_h\r_h+\l \grad((\tilde{\m}_h^{n+1} - \m_h^n)\cdot\m_h^n),\grad(\m_h^n\cdot\v_h)\r_h\notag\\
&-\l(\tilde{\m}_h^{n+1} - \m_h^n) \cdot\grad\m_h^n,\grad(\m_h^n\cdot\v_h)\r_h\notag\\
&+\l\grad (\tilde{\m}_h^{n+1} - \m_h^n),\grad( I_h(\m_h^n(\m_h^n\cdot\v_h)) - \m_h^n(\m_h^n\cdot\v_h))\r_h.
\end{align}

Additionally, we show that the scheme \eqref{a2}--\eqref{a8} is also equivalent to the mass-lumped FEM discretization of the normalized tangent plane scheme. Indeed, for a nodal unit vector field \(\m_h^n\), define
\[
T_h(\m_h^n)
:=
\{\w_h\in\mathbf S_h^1:
\w_h(\z)\cdot \m_h^n(\z)=0,\ \forall \z\in\mathcal Q\}.
\]
Let
\[
P_{\m_h^n}^{\top}\v_h
:=
\v_h-I_h\bigl(\m_h^n(\m_h^n\cdot\v_h)\bigr).
\]
Then \(P_{\m_h^n}^{\top}\) is self-adjoint with respect to
\((\cdot,\cdot)_h\), and it is the nodal projection onto \(T_h(\m_h^n)\).
Using the self-adjointness of \(P_{\m_h^n}^{\top}\), \eqref{c6} and \eqref{o6}, \eqref{a2} can be rewritten as follows:
\begin{align*}
&\l \frac{\tilde{\m}_h^{n+1} - \m_h^{n}}{\Delta t},P_{\m_h^n}^\top\v_h \r_h = \l P_{\m_h^n}^\top\frac{\tilde{\m}_h^{n+1} - \m_h^{n}}{\Delta t},\v_h \r_h =\l \frac{\tilde{\m}_h^{n+1} - \m_h^{n}}{\Delta t},\v_h \r_h \\
&=\l\tilde\Delta_h \tilde{\m}_h^{n+1},\v_h - I_h(\m_h^n(\m_h^n\cdot\v_h))\r_h =  \l\tilde\Delta_h \tilde{\m}_h^{n+1},P_{\m_h^n}^\top\v_h\r_h.
\end{align*} 
Using the nodal orthogonality \eqref{c6}, we have $\u_h^{n+1}\in T_h(\m_h^n)$ for $\u_h^{n+1}\coloneqq\frac{\widetilde{\m}_h^{n+1}-\m_h^n}{\Delta t}$.
Then the scheme \eqref{a2}--\eqref{a8} is equivalent to the following tangent plane scheme: Find $\u_h^{n+1}\in T_h(\m_h^n)$ such that for all $\w_h\in T_h(\m_h^n)$, there holds
\begin{align*}
\l& \u_h^{n+1},\w_h \r_h + \Delta t\l\grad \u_h^{n+1},\grad\w_h\r_h+ \l\grad\m_h^n,\grad\w_h\r_h = 0,\\
&\m_h^{n+1} = I_h\l\frac{\m_h^n+\Delta t\u_h^{n+1}}{|\m_h^n+\Delta t\u_h^{n+1}|}\r.
\end{align*}
Thus, the present analysis also provides results for the normalized
tangent plane scheme, which is widely used for the numerical
approximation of the harmonic map heat flow and the
Landau--Lifshitz--Gilbert equation.

\section{Main results}\label{sec3}
Our first result is the unconditional energy decay for scheme \eqref{a2}--\eqref{a8}. As mentioned before, for a projection-based scheme to yield energy decay, the stability of the projection step is needed.
The stability results for the projection step established in
\cite{MR2177142, MR3454357} concern piecewise affine finite elements on
simplicial meshes and rely essentially on the fact that their gradients
are constant on each element. This property, however, does not carry over to tensor-product $Q_1$ elements since the gradient of a $Q_1$ function is still a polynomial within each element. Consequently, the Lipschitz continuity of the normalization mapping $\{\s\in \mathbb{R}^3:|\s|\ge 1\}\to \mathbb{R}^3, \s\mapsto \s/|\s|$ cannot be directly applied to control the Dirichlet energy of the projected function. To circumvent this difficulty, we resort to the discrete energy norm using \eqref{e6}. Although the main analysis of this work is carried out on Cartesian rectangular/cuboidal tensor-product meshes, the following projection stability estimate is established on a slightly more general class of meshes.
\begin{lemma}[Stability of the nodal projection]
Let $\mathcal T_h$ be a shape-regular quadrilateral or hexahedral mesh.
Then there exists a constant $C_{\mathcal T_h}>0$ independent of $h$, such that for all $\tilde\v_h\in \S_h^1$ with $|\tilde\v_h(\z)|\ge1$ for all $\z\in \mathcal{Q}$, we have
  \begin{align*}
    &\norm{\grad \v_h}_{h}^2\le C_{\mathcal{T}_h}\norm{\grad \tilde\v_h}_h^2,
  \end{align*}
  where $\v_h = I_h\l \dfrac{\tilde\v_h}{|\tilde\v_h|} \r$. 
  In particular, if $\mathcal{T}_h$ is a Cartesian rectangular/cuboidal tensor-product mesh, then
\begin{equation}
    \norm{\grad \v_h}_h^2\le\norm{\grad \tilde\v_h}_h^2.\label{e8}
  \end{equation}
\end{lemma}
\begin{proof}
We prove the estimate element by element. Let $K\in\mathcal T_h$ be fixed.
For the reference element $\hat K=[0,1]^d$, let
$\Phi_K:\hat K\to K$ be the geometric mapping with Jacobian matrix
$J_K$. For any finite element function $\w_h\in\S_h^1$, set
\[
  \hat\w_h=(\w_h|_K)\circ\Phi_K .
\]

We first prove the stability of the nodal projection on the reference
element. On the reference element $\hat{K}$, we have
\[
  \partial_{i}\hat\v_h
  (\z_{\hat K,i,\alpha'}^\pm) = \mathfrak D_{\hat K,i}^{\alpha'}\hat\v_h = \frac{\hat\v_h(\z_{\hat K,i,\alpha'}^+) -
  \hat\v_h(\z_{\hat K,i,\alpha'}^-)}{h_{K,i}} = \hat\v_h(\z_{\hat K,i,\alpha'}^+) - \hat\v_h(\z_{\hat K,i,\alpha'}^-).
\]
Since
$\hat\v_h(\z)=\l I_h\frac{\hat{\tilde\v}_h}{|\hat{\tilde\v}_h|}\r(\z) = \frac{\hat{\tilde\v}_h(\z)}{|\hat{\tilde\v}_h(\z)|}$ for all $\hat\z\in\mathcal Q_{\hat K}$, we have by the Lipschitz continuity of the normalization mapping $\{\s\in \mathbb{R}^3:|\s|\ge 1\}\to \mathbb{R}^3, \s\mapsto \s/|\s|$
\begin{align}\label{o8a}
  \left|
   \partial_{i}\hat\v_h
  (\z_{\hat K,i,\alpha'}^\pm)
  \right|
  &=
  \left|
  \frac{\hat{\tilde\v}_h(\z_{\hat K,i,\alpha'}^+)}{|\hat{\tilde\v}_h(\z_{\hat K,i,\alpha'}^+)|}
  -
  \frac{\hat{\tilde\v}_h(\z_{\hat K,i,\alpha'}^-)}{|\hat{\tilde\v}_h(\z_{\hat K,i,\alpha'}^-)|}
  \right| \notag\\
  &\le
  \left|
  \hat{\tilde\v}_h(\z_{\hat K,i,\alpha'}^+)
  -
  \hat{\tilde\v}_h(\z_{\hat K,i,\alpha'}^-)
  \right|
  = \left|
  \partial_{i}\hat{\tilde\v}_h
  (\z_{\hat K,i,\alpha'}^\pm)
  \right|.
\end{align}
Squaring both sides, multiplying by the quadrature weights, and summing over all $i$-edges and all directions $i$, we have
\begin{align}
  \norm{\grad\hat\v_h}_{L_h^2(\hat K)}^2\le \norm{\grad\hat{\tilde\v}_h}_{L_h^2(\hat K)}^2 .
  \label{o8b}
\end{align}

We next derive the transformation of the discrete energy norm under the
geometric mapping $\Phi_K$. For any function $\w_h\in\S_h^1$, by the definition of the broken discrete gradient norm, we have
\[
\begin{aligned}
  \norm{\grad \w_h}_{L_h^2(K)}^2&=
  \sum_{\z\in\mathcal Q_{K}}
  \omega_{K,\z}
  \left|
  \grad(\w_h|_K)(\z)
  \right|^2  \\
  &=
  \sum_{\z\in\mathcal Q_{\hat K}}
  \omega_{\hat K,\z}
  |\det J_K(\z)|
  \left|
  J_K(\z)^{-\top}\grad\hat\w_h(\z)
  \right|^2  \\
  &\le
  \norm{\det J_K}_{L^\infty(\hat K)}
  \norm{J_K^{-\top}}_{L^\infty(\hat K)}^2
  \norm{\grad\hat\w_h}_{L_h^2(\hat K)}^2 .
\end{aligned}
\]
By the shape-regularity of the family of quadrilateral or hexahedral
meshes, there exists a geometry constant $\sigma_\sharp>0$, independent
of $h$, such that \cite{MR4242224}
\begin{align*}
  &\norm{J_K}_{L^\infty(\hat K)}
  \le C\sigma_\sharp h_K,
  \qquad
  \norm{J_K^{-1}}_{L^\infty(\hat K)}
  \le C\sigma_\sharp h_K^{-1},\\
  &\norm{\det J_K}_{L^\infty(\hat K)}
  \le C\sigma_\sharp h_K^d,
  \qquad
  \norm{\det (J_K^{-1})}_{L^\infty(\hat K)}
  \le C\sigma_\sharp h_K^{-d}.
\end{align*}
Consequently,
\[
  \norm{\grad \w_h}_{L_h^2(K)}^2
  \le
  C h_K^{d-2}
  \norm{\grad\hat\w_h}_{L_h^2(\hat K)}^2 .
\]
We similarly obtain
\[
\begin{aligned}
  &\norm{\grad\hat\w_h}_{L_h^2(\hat K)}^2
  =
  \sum_{\z\in\mathcal Q_{K}}
  \omega_{K,\z}
  |\det (J_K^{-1})(\z)|
  \left|
  J_K(\z)^{\top}\grad(\w_h|_K)(\z)
  \right|^2  \\
  &\le
  \norm{\det (J_K^{-1})}_{L^\infty(\hat K)}\norm{J_K^\top}_{L^\infty(\hat K)}^2
  \norm{\grad(\w_h|_K)}_{L_h^2(K)}^2\le
  C h_K^{2-d}
  \norm{\grad\w_h}_{L_h^2(K)}^2 .
\end{aligned}
\]
Hence there exist constants $c_0,c_1>0$ independent of $h_{K}$, such that
\[
  c_0 h_K^{d-2}
  \norm{\grad\hat\w_h}_{L_h^2(\hat K)}^2
  \le
  \norm{\grad\w_h}_{L_h^2(K)}^2
  \le
  c_1 h_K^{d-2}
  \norm{\grad\hat\w_h}_{L_h^2(\hat K)}^2 .
\]

Combining this equivalence with \eqref{o8b}, we obtain
\begin{align*}
  \norm{\grad \v_h}_{L_h^2(K)}^2
  \le
  c_1 h_K^{d-2}
  \norm{\grad\hat\v_h}_{L_h^2(\hat K)}^2 \le
  c_1 h_K^{d-2}
  \norm{\grad\hat{\tilde\v}_h}_{L_h^2(\hat K)}^2 \le
  \frac{c_1}{c_0}
  \norm{\grad\tilde\v_h}_{L_h^2(K)}^2 .
\end{align*}
Thus,
\[
  \norm{\grad \v_h}_{L_h^2(K)}^2
  \le
  C_K
  \norm{\grad \tilde\v_h}_{L_h^2(K)}^2,
\]
where $C_K$ depends only on the geometry of the mapping $\Phi_K$, but is
independent of $h$. Summing over all elements gives
\[
  \norm{\grad \v_h}_{L_h^2}^2
  \le
  C_{\mathcal T_h}
  \norm{\grad \tilde\v_h}_{L_h^2}^2 .
\]

If $\mathcal T_h$ consists of Cartesian rectangles or cuboids, i.e.,
\[
  K=\prod_{i=1}^d[x_{K,i}^-,x_{K,i}^+],
  \qquad
  h_{K,i}=x_{K,i}^+-x_{K,i}^-,
\]
then the mapping $\Phi_K$ is affine with a constant diagonal Jacobian. Specifically, 
\[
  \Phi_K(\hat\x)=(x_{K,1}^-,\ldots,x_{K,d}^-)^\top +J_K\hat\x,
  \qquad
  J_K=\operatorname{diag}(h_{K,1},\ldots,h_{K,d}).
\]
Hence
\[
  J_K^{-\top}
  =
  \operatorname{diag}(h_{K,1}^{-1},\ldots,h_{K,d}^{-1}),
  \qquad
  |\det J_K|
  =
  \prod_{i=1}^d h_{K,i}.
\]
Using the directional pointwise estimate \eqref{o8a}, we get
\begin{align*}
  \norm{\grad\v_h}_{L_h^2(K)}^2
  &=
  \sum_{\z\in\mathcal Q_{\hat K}}
  \omega_{\hat K,\z}
  |\det J_K|
  \sum_{i=1}^d
  h_{K,i}^{-2}
  \left|
  \partial_i\hat\v_h(\z)
  \right|^2 \\
  &\le
  \sum_{\z\in\mathcal Q_{\hat K}}
  \omega_{\hat K,\z}
  |\det J_K|
  \sum_{i=1}^d
  h_{K,i}^{-2}
  \left|
  \partial_i\hat{\tilde\v}_h(\hat\z)
  \right|^2 =
  \norm{\grad\tilde\v_h}_{L_h^2(K)}^2 .
\end{align*}
Summing over all rectangular or cuboidal elements gives \eqref{e8}.
\end{proof}
  \begin{remark}
    The key reason why \(C_{\mathcal T_h}=1\) is that the Jacobian \(J_K\)
    is constant and diagonal, so the geometric transformation factors cancel
    exactly. More generally, the same argument applies whenever
    \(J_K=Q_KD_K\), with \(Q_K^\top Q_K=I\) and \(D_K\) diagonal with positive
    entries; this includes rotated rectangles and cuboids. One can also prove \eqref{e8} by the edge-difference representation \eqref{o1} to avoid the reference element altogether. Another proof in the context of finite differences can be found in \cite{li2026stabilityerroranalysisfully}.
  \end{remark}

We next prove the unique solvability and unconditional discrete energy decay of the proposed scheme.
  \begin{theorem}[Unique solvability and discrete energy decay]
  The numerical scheme \eqref{a2}--\eqref{a8} has a unique solution and satisfies the following stability estimate:
    \begin{align}
      &\frac{1}{2}\left\| \grad \m_h^{n+1} \right\|_h^2 + \Delta t\sum_{k=0}^{n}\norm{\m_h^k\times\tilde\Delta_h\tilde\m_h^{k+1}}_h^2\le \frac{1}{2}\left\| \grad \m_h^{0} \right\|_h^2.\label{e3}
    \end{align}
  \end{theorem}
  \begin{proof}
    Since the scheme is a finite-dimensional linear system, it suffices to show that the associated homogeneous problem has only the trivial solution. The corresponding homogeneous system of \eqref{a2} is
  \begin{align}
  \l \frac{\tilde{\m}_h^{n+1}}{\Delta t},\v_h \r_h=& -\l\m_h^n\times\l\m_h^n\times\tilde\Delta_h \tilde{\m}_h^{n+1}\r,\v_h\r_h.
  \end{align}
  Taking $\v_h = -\tilde\Delta_h \tilde{\m}_h^{n+1}$, we obtain
  \begin{align*}
  &\frac{\norm{\grad\tilde{\m}_h^{n+1}}_h^2}{\Delta t}+\norm{\m_h^n\times\tilde\Delta_h \tilde{\m}_h^{n+1}}_h^2=0,
  \end{align*}
  which means that $\tilde\m_h^{n+1}$ is constant, hence $\tilde\Delta_h \tilde{\m}_h^{n+1} = 0$, and thus $(\tilde\m_h^{n+1}, \v_h)_h = 0$ holds for all $\v_h\in\S_h^1$. Taking $\v_h = \tilde\m_h^{n+1}$, we obtain the associated homogeneous problem has only the trivial solution.

  Substituting $\v_h = -\tilde\Delta_h\tilde\m_h^{n+1}$ into \eqref{a2}, we obtain
  \begin{align*}
  &\frac{1}{2\Delta t}\l\norm{\grad \tilde{\m}_h^{n+1}}_h^2 - \norm{\grad \m_h^n}_h^2\r + \norm{\m_h^n\times\tilde\Delta_h\tilde\m_h^{n+1}}_h^2 \le 0.
  \end{align*}
  Combining this estimate with \eqref{e8}, we obtain
  \begin{align*}
  &\frac{1}{2\Delta t}\l\norm{\grad {\m}_h^{n+1}}_h^2 - \norm{\grad \m_h^n}_h^2\r + \norm{\m_h^n\times\tilde\Delta_h\tilde\m_h^{n+1}}_h^2 \le 0,
  \end{align*}
  Summing the resulting inequality over \(k=0,1,\cdots,n\) gives \eqref{e3}.
  \end{proof}
We now present the unconditional optimal-order convergence rates of our scheme.
\begin{theorem}\label{th1}
Suppose that the initial-boundary value problem \eqref{a7}--\eqref{c4} admits a unique solution $\m$ with the following regularity assumptions:
\begin{align*}
  &\m^0\in \mathbf{W}^{2,4},\ \m\in L^\infty(0,T;\mathbf{W}^{2,4}\cap\H^3),\ \m_t\in L^\infty(0,T;\mathbf{H}^{2}), \ \m_{tt}\in L^\infty(0,T;\mathbf{L}^{2}).
\end{align*}
Then there exist two constants $0<h_0, \Delta t_0<1$ such that when $\Delta t\le\Delta t_0$, $h\le h_0$, the numerical scheme \eqref{a2}--\eqref{a8} yields a unique solution $\m_h^{l}\in \S_h^1$, $l=1,\cdots,N$, with the following error bounds
\begin{align*}
  &\max_{1\le l\le N} \norm{\m_h^{l}-\m(t_l)}^2\le  C(\Delta t^2+h^{4}),\quad \Delta t\sum_{l=1}^{N}\norm{\m_h^l - \m(t_l)}_{H^1}^2\le C(\Delta t^2+h^{2}),
\end{align*}
where $C$ is a constant independent of $\Delta t$ and $h$.
\end{theorem}
\section{Proof of Theorem \ref{th1}}\label{sec4} 
In this section, we present the proof of Theorem \ref{th1}. We first present some technical lemmas for the later estimates, then give a bound for the truncation error. 
The main proof is based on a bootstrap induction. We first obtain the
optimal error estimate for the auxiliary solution from the error equation,
the projected-error contraction, and the discrete Gronwall inequality.
The remaining bootstrap bound concerns the discrete Laplacian of the
auxiliary error. It is first estimated from the error equation in dual
form, using the already established error bound, the discrete Sobolev
inequalities, and Lemma~\ref{lemma8}. A complementary estimate is then
obtained from the inverse inequality and the \(L^2\)-error bound. Taking
the minimum of these two estimates closes the bootstrap uniformly and
avoids any coupling condition between the time step and the mesh size.  Once the induction is closed, the desired estimates for
the projected solution follow from the contraction property and standard
interpolation estimates.

\subsection{Some useful lemmas}
The following superconvergence result will be used to bound the truncation error, the proof is provided in Appendix \ref{u6}.
\begin{lemma}\label{lemma11} 
  Let \(\m\) be the exact solution of \eqref{a7}--\eqref{c4} satisfying
the regularity assumptions of Theorem~\ref{th1}, and let
\(\v_h\in\S_h^1\). Then the following superconvergence result holds for the Lagrange interpolation operator $I_h$
  \[ 
  \left| \l \grad I_h\m(t_{n+1}) , \grad \v_h\r_h-\l \grad \m(t_{n+1}), \grad \v_h \r\right| \le C h^2\norm{\v_h}_{H^1}. 
  \]
\end{lemma}
The following lemma presents the quadrature error between the discrete inner product and the standard inner product for the tensor product element, which is a consequence of the Bramble--Hilbert lemma \cite{MR263214}.
\begin{lemma}{\rm(\cite[Lemma 3.6]{MR4377027})}\label{lemma1}
  Let $K\in \mathscr{T}_h$, and denote by $(\cdot, \cdot)_K$ the $L^2$ inner product on $K$. Let $V_h$ be the space of polynomials of some degree $l\ge 0$ on $K$. Then
  \begin{equation}\label{r7}
    \left|\l 1, I_Kf-f \r_{K}\right|\leq Ch_K^{2}\sum_{i=1}^d\norm{\partial_{i}^{2}f}_{L^1(K)}\quad \forall f\in V_h,
  \end{equation}
  where $C$ is a positive constant independent of $h$ and $f$ (but may depend on $l$).
\end{lemma}
The following elementary estimates will be used to control
quadrature errors in the nonlinear terms.
\begin{lemma}
Let \(\a_h,\b_h,\mathbf c_h\in \S_h^1\), and define
\[
 \F^0_h:=\a_h\cdot\b_h,
 \qquad
 \F^1_h:=\a_h(\b_h\cdot\mathbf c_h).
\]
Then, for each \(K\in\mathscr T_h\), one has
$
  \F^0_h|_K\in Q_2(K),
  \F^1_h|_K\in Q_3(K).
$
Moreover, along each \(x_i\)-direction edge of \(K\), the following
one-dimensional interpolation estimates hold for every
\(\z\in\mathcal Q_K\):
\begin{align}
  \left|
  \partial_i(I_h\F^0_h-\F^0_h)|_K(\z)
  \right|
  \le&
  \frac{h_{K,i}}{2}
  \left|
  \partial_{ii}\F^0_h|_K(\z)
  \right|,
  \label{y6}\\
  \left|
  \partial_i(I_h\F^1_h-\F^1_h)|_K(\z)
  \right|
  \le&
  \frac{h_{K,i}}{2}
  \left|
  \partial_{ii}\F^1_h|_K(\z)
  \right|
  +
  \frac{h_{K,i}^2}{6}
  \left|
  \partial_{iii}\F^1_h|_K(\z)
  \right|.
  \label{u7}
\end{align}
\end{lemma}
\begin{proof}
Restricting \(\F_h^0\) and \(\F_h^1\) to an \(x_i\)-direction edge of
\(K\), we obtain one-dimensional polynomials of degree at most \(2\) and
\(3\), respectively. The derivative of their linear interpolants is the
difference quotient between the two endpoints of the edge. Comparing this
constant derivative with the derivative of the original polynomial at an
endpoint gives \eqref{y6} and \eqref{u7}; the vector-valued estimate is
obtained componentwise.
\end{proof}
The following two lemmas establish discrete Sobolev inequalities and the relation between two discrete Laplacian operators, which will be used in the completion of the induction argument. The proof of Lemma \ref{lemma3} can be found in the Appendix \ref{y5}.
  \begin{lemma}{\rm \cite{MR4377027}}
    The discrete Laplacian operator $\Delta_h:\S_h^1\to \S_h^1$ is defined by duality
    \begin{equation}\label{b4}
    (\Delta_h\v_h, \w_h) = -(\grad\v_h,\grad\w_h), \quad \forall\v_h, \w_h\in \S_h^1.
    \end{equation}
  Then the following discrete Sobolev interpolation and embedding inequalities hold:
    \begin{align}
      \norm{\v_h}_{L^\infty}&\le C\norm{\v_h}^{1-\frac{d}{4}}(\norm{\v_h}+\norm{\Delta_h\v_h})^{\frac{d}{4}},\label{b5}\\
      \norm{\grad\v_h}_{L^6}&\le C\norm{\Delta_h\v_h}.\label{b6}
    \end{align}
  \end{lemma}
\begin{lemma}\label{lemma3}
  Let \(\mathscr T_h\) be a Cartesian rectangular/cuboidal tensor-product
  mesh satisfying shape-regularity and quasi-uniformity. The two discrete Laplacian operators $\tilde\Delta_h$ and $\Delta_h$ have the following relation for all $\v_h\in\S_h^1$
  \begin{equation}\label{p2}
    \|\Delta_h\v_h\|
    \le C_5\|\tilde\Delta_h\v_h\|_h .
  \end{equation}
\end{lemma}
\subsection{Truncation error estimates}
Now we will estimate the truncation error.
For an arbitrary $t_{n+1}$, the exact solution of the problem \eqref{a7}--\eqref{c4} satisfies the following equation for all $\v_h\in \S_h^1$:
	\begin{align*}
		&(\partial_t\m(t_{n+1}),\v_h) + \l \grad \m(t_{n+1}), \grad \v_h \r = \l |\grad \m(t_{n+1})|^2\m(t_{n+1}), \v_h \r,
	\end{align*}
	which can be rewritten as
	\begin{align}
	&\l \frac{I_h\m(t_{n+1})-I_h\m(t_{n})}{\Delta t},\v_h \r_h+\l \grad I_h\m(t_{n+1}) , \grad \v_h\r_h\notag\\
  =&\l\grad I_h\m(t_{n})\cdot\grad I_h\m(t_{n+1}) , I_h\m(t_{n})\cdot\v_h \r_h + \Ep(\v_h)\notag\\
  &-\l (I_h\m(t_{n+1}) - I_h\m(t_{n}))\cdot\grad I_h\m(t_{n}), \grad( I_h\m(t_{n})\cdot\v_h)\r_h\notag\\
  &+\l\grad I_h(\m(t_{n+1}) - \m(t_n)),\grad I_h(I_h\m(t_{n})(I_h\m(t_{n})\cdot\v_h))\r_h\notag\\
  &-\l\grad I_h(\m(t_{n+1}) - \m(t_n)),\grad(I_h\m(t_{n})(I_h\m(t_{n})\cdot\v_h))\r_h.\label{a1}
	\end{align}
Here, the truncation error $\Ep(\v_h):\S_h^1\to \mathbb{R}$ is defined as 
\begin{align*}
  &\Ep(\v_h) \\
  =& \l \frac{I_h\m(t_{n+1})-I_h\m(t_{n})}{\Delta t}, \v_h\r_h - (\partial_t\m(t_{n+1}),\v_h)\\
  &+\l \grad I_h\m(t_{n+1}) , \grad \v_h\r_h-\l \grad \m(t_{n+1}), \grad \v_h \r\\
  &+\l |\grad \m(t_{n+1})|^2\m(t_{n+1}), \v_h \r - \l \grad I_h\m(t_{n})\cdot\grad I_h\m(t_{n+1}) , I_h\m(t_{n})\cdot\v_h \r_h\\
  &+\l I_h(\m(t_{n+1})- \m(t_{n}))\cdot\grad I_h\m(t_{n}), \grad(I_h\m(t_{n})\cdot\v_h)\r_h\\
  &-\l\grad I_h(\m(t_{n+1}) - \m(t_n)),\grad I_h(I_h\m(t_{n})(I_h\m(t_{n})\cdot\v_h))\r_h\\
  &+\l\grad I_h(\m(t_{n+1}) - \m(t_n)),\grad(I_h\m(t_{n})(I_h\m(t_{n})\cdot\v_h))\r_h\\
  \eqqcolon&\Ep_1(\v_h)+\Ep_2(\v_h)+\Ep_3(\v_h)+\Ep_4(\v_h).
\end{align*}
We have the following bound for $\Ep(\v_h)$.
\begin{lemma}
  Under the regularity assumptions of Theorem \ref{th1}, the truncation error $\Ep(\v_h)$ satisfies for all $\v_h\in \S_h^1$
  \begin{align}
    |\Ep(\v_h)|\le C(\Delta t^2+h^4+\norm{\v_h}^2)+\epsilon\norm{\grad \v_h}^2.\label{a4}
  \end{align}
\end{lemma}
\begin{proof}
  The estimate for $\Ep_1(\v_h)$ and $\Ep_2(\v_h)$ can be found in \cite[Section 3.2]{MR4377027} and Lemma \ref{lemma11}, as a result, we have
\begin{align*}
  |\Ep_1(\v_h)|+|\Ep_2(\v_h)|\le C\l\Delta t+h^{2}\r\norm{\v_h}+Ch^{2} \norm{\v_h}_{H^1}.
\end{align*}
For elementwise smooth functions $f$ and $g$, we denote
\[
\Ep_{\rm quad}(f, g) \coloneqq (f,g)_h - (f,g).
\]
Then $\Ep_3(\v_h)$ can be split as follows:
\begin{align*}
  \Ep_3(\v_h) 
  =\ &\l |\grad \m(t_{n+1})|^2, \m(t_{n+1})\cdot\v_h \r - \l \grad I_h\m(t_{n})\cdot\grad I_h\m(t_{n+1}),I_h\m(t_{n})\cdot \v_h \r\\
  &-\Ep_{\rm quad}\l \grad I_h\m(t_{n})\cdot\grad I_h\m(t_{n+1}),I_h\m(t_{n})\cdot \v_h\r\\
  \eqqcolon & \Ep_{3,1}(\v_h) + \Ep_{3,2}(\v_h).
\end{align*}
Using integration by parts and the interpolation error estimates \cite[Theorem 4.4.20]{brenner2008mathematical}, we have
\begin{align*}
  &|\Ep_{3,1}(\v_h)| \\
  =& \left|\l \grad (\m(t_{n+1}) - I_h\m(t_{n+1}))\cdot\grad \m(t_{n+1}), \m(t_{n+1})\cdot\v_h \r\right|\\
  &+ \left|\l \grad (I_h\m(t_{n+1}) - I_h\m(t_{n}))\cdot\grad \m(t_{n+1}), \m(t_{n+1})\cdot\v_h \r\right|\\
  &+  \left|\l\grad (I_h\m(t_{n}) - \m(t_{n}))\cdot\grad (\m(t_{n+1}) - I_h\m(t_{n+1})), \m(t_{n+1})\cdot\v_h \r\right|\\
  &+  \left|\l\grad \m(t_{n})\cdot\grad (\m(t_{n+1})- I_h\m(t_{n+1})), \m(t_{n+1})\cdot\v_h \r\right|\\ 
  &+ \left|\l\grad I_h\m(t_{n})\cdot\grad I_h\m(t_{n+1}), (\m(t_{n+1}) - I_h\m(t_{n}))\cdot\v_h \r\right|\\
  \le&\left|\l( \m(t_{n+1})- I_h\m(t_{n+1}))\cdot\Delta \m(t_{n+1}), \m(t_{n+1})\cdot\v_h \r\right|\\
  &+\left|\l( \m(t_{n+1})- I_h\m(t_{n+1}))\cdot\grad \m(t_{n+1}), \grad(\m(t_{n+1})\cdot\v_h )\r\right|\\
  &+C\norm{\grad(I_h\m(t_{n+1}) - I_h\m(t_n))}\norm{\grad \m(t_{n+1})}_{L^\infty}\norm{\m(t_{n+1})}_{L^\infty}\norm{\v_h}\\
  &+C\norm{\grad (I_h\m(t_{n}) - \m(t_{n}))}_{L^4}\norm{\grad (\m(t_{n+1}) - I_h\m(t_{n+1}))}_{L^4}\norm{\m(t_{n+1})}_{L^\infty}\norm{\v_h}\\
  &+\left|\l\Delta\m(t_{n})\cdot( \m(t_{n+1})- I_h\m(t_{n+1})), \m(t_{n+1})\cdot\v_h \r\right|\\
  &+\left|\l\grad\m(t_{n})\cdot(\m(t_{n+1})- I_h\m(t_{n+1})), \grad(\m(t_{n+1})\cdot\v_h)\r\right|\\
  &+C\norm{\grad I_h\m(t_n)}_{L^\infty}\norm{\grad I_h\m(t_{n+1})}_{L^\infty}\norm{\m(t_{n+1}) - I_h\m(t_{n})}\norm{\v_h}\\
  \le& C\l \Delta t+h^{2}\r\norm{\v_h}+Ch^{2}\norm{\grad\v_h}.
\end{align*}
Using the quadrature error bound in Lemma \ref{lemma1}, and noting that functions in $\S_h^1$ have at most first-order nonzero partial derivatives in each variable, we have
\begin{align*}
  \Ep_{3,2}(\v_h) 
  \le& Ch^{2}\sum_{K\in\mathscr{T}_h}\sum_{i=1}^d\norm{\partial_i^{2}\l (\grad I_h\m(t_{n})\cdot\grad I_h\m(t_{n+1})) I_h\m(t_{n})\cdot\v_h \r}_{L^1(K)}\\
  \le& Ch^{2}\sum_{K\in\mathscr{T}_h}\norm{\m(t_{n})}_{W^{2,4}(K)}\norm{\m(t_{n+1})}_{W^{2,4}(K)}\norm{\m(t_{n})}_{W^{1,\infty}(K)}\norm{\v_h}_{H^1(K)}\\
  \le & Ch^{2}\norm{\v_h}_{H^1}.
\end{align*}
We thus have 
\[
|\Ep_3(\v_h)|\le C\l \Delta t+h^{2}\r\norm{\v_h}+Ch^{2}\norm{\v_h}_{H^1}.
\]
Finally, split $\Ep_4(\v_h)$ as follows:
\begin{align*}
  \Ep_4(\v_h)
  =&\l \grad \l(I_h\m(t_{n+1})- I_h\m(t_{n}))\cdot I_h\m(t_{n})\r, \grad(I_h\m(t_{n})\cdot\v_h)\r_h\\
  &+\l\grad (I_h\m(t_{n+1}) - I_h\m(t_n)),(I_h\m(t_{n})\cdot\v_h)\cdot\grad I_h\m(t_{n}) \r_h\\
  &-\l\grad (I_h\m(t_{n+1}) - I_h\m(t_n)),\grad I_h(I_h\m(t_{n})(I_h\m(t_{n})\cdot\v_h))\r_h\\
  \eqqcolon&\Ep_{4,1}(\v_h)+\Ep_{4,2}(\v_h)+\Ep_{4,3}(\v_h)
\end{align*}
By the Lemma \ref{lemma1}, integration by parts, the interpolation error estimates and the inverse inequality, we have
\begin{align*}
  &\left|\l\tilde\Delta_h I_h(\m(t_{n+1}) - \m(t_{n})), \w_h\r_h\right| = \left|\l \grad I_h(\m(t_{n+1}) - \m(t_{n})), \grad\w_h \r_h\right|\\
  \le&\left|\Ep_{\rm quad}\l \grad I_h(\m(t_{n+1}) - \m(t_{n})), \grad\w_h \r\right|+\left|\l \grad (\m(t_{n+1}) - \m(t_{n})), \grad\w_h \r\right|\\
  &+\left|\l \grad (I_h(\m(t_{n+1}) - \m(t_{n})) - (\m(t_{n+1}) - \m(t_{n}))), \grad\w_h \r\right|\\
  \le&Ch^{2}\sum_{K\in\mathscr{T}_h}\sum_{i=1}^d\sum_{j\ne i}\norm{\partial_{ij}I_h(\m(t_{n+1}) - \m(t_{n}))}_{L^2(K)}\norm{\partial_{ij}\w_h }_{L^2(K)}\\
  &+\left|\l \Delta (\m(t_{n+1}) - \m(t_{n})), \w_h \r\right|+Ch\norm{\m(t_{n+1}) - \m(t_{n})}_{H^2}\norm{\grad\w_h}\\
  \le&C\norm{\m(t_{n+1}) - \m(t_{n})}_{H^2}\norm{\w_h} 
  \le C\Delta t\norm{\m_t}_{L^\infty(0,T;H^2)}\norm{\w_h},
\end{align*}
which combines the equivalence \eqref{b3} and take the supremum, imply that 
\begin{align}\label{y7}
\norm{\tilde\Delta_h I_h(\m(t_{n+1}) - \m(t_{n}))}_h\le C\Delta t.
\end{align}
We can similarly derive that
\begin{align}\label{y8}
\norm{\tilde\Delta_h I_h((\m(t_{n+1}) - \m(t_{n}))\cdot\m(t_n))}_h\le C\norm{\m_t}_{L^\infty(0,T;H^2)}\norm{\m(t_{n})}_{W^{2,4}}\Delta t.
\end{align}
Note that for all $\z\in\mathcal{Q}$, we have
\[
I_h\l I_h(\m(t_{n+1})- \m(t_{n}))\cdot I_h\m(t_{n})\r(\z) = I_h((\m(t_{n+1})- \m(t_{n}))\cdot\m(t_{n}))(\z).
\]
Thus, using \eqref{y6} twice with $\F^0_h = I_h(\m(t_{n+1})- \m(t_{n}))\cdot I_h\m(t_{n})$ and $\F^0_h = I_h\m(t_{n})\cdot\v_h$ respectively, we have  
\begin{align*}
  &\Ep_{4,1}(\v_h)\\
  =&\l \grad (I_h(\m(t_{n+1})- \m(t_{n}))\cdot I_h\m(t_{n})), \grad (I_h\m(t_{n})\cdot\v_h)\r_h\\
  &-\l \grad I_h((\m(t_{n+1})- \m(t_{n}))\cdot\m(t_{n})), \grad (I_h\m(t_{n})\cdot\v_h)\r_h\\
  &+\l \grad I_h((\m(t_{n+1})- \m(t_{n}))\cdot\m(t_{n})), \grad ((I_h\m(t_{n})\cdot\v_h) - I_h(I_h\m(t_{n})\cdot\v_h))\r_h\\
  &-\l \tilde\Delta_h I_h((\m(t_{n+1})- \m(t_{n}))\cdot\m(t_{n})),  I_h\m(t_{n})\cdot\v_h\r_h\\
  \le&\sum_{i=1}^d\l \frac{h}{2}|\partial_{ii}(I_h(\m(t_{n+1})- \m(t_{n}))\cdot I_h\m(t_{n}))|, |\partial_i(I_h\m(t_{n})\cdot\v_h)|\r_{h}\\
  &+\sum_{i=1}^d\l|\partial_iI_h((\m(t_{n+1})- \m(t_{n}))\cdot\m(t_{n}))|, \frac{h}{2}|\partial_{ii}(I_h\m(t_{n})\cdot\v_h)|\r_{h}\\
  &+C\norm{\tilde\Delta_hI_h((\m(t_{n+1})- \m(t_{n}))\cdot\m(t_{n}))}_{h}\norm{\v_h}_h\\
  \le&C\Delta t\norm{\v_h}_h,
\end{align*}
where \eqref{y8} is used.
Furthermore, using \eqref{y7} again, we obtain
\begin{align*}
  &\Ep_{4,2}(\v_h)+\Ep_{4,3}(\v_h)\\
  \le& C\norm{\grad I_h(\m(t_{n+1}) - \m(t_n))}\norm{I_h\m(t_{n})}_{L^\infty}\norm{\v_h}\norm{\grad I_h\m(t_{n})}_{L^\infty}\\
  &+\norm{\tilde\Delta_h I_h(\m(t_{n+1}) - \m(t_n))}_h\norm{I_h\m(t_{n})}_{L^\infty}^2\norm{\v_h}_h\le C\Delta t\norm{\v_h}_h.
\end{align*}
As a result, we have 
\begin{align*}
  \Ep_{4}(\v_h)\le C\Delta t\norm{\v_h}_h.
\end{align*}
Combining the above estimates with \eqref{p7}, using \eqref{b3}--\eqref{p7} again, we obtain
\begin{equation}\label{c5}
|\Ep(\v_h)|\le C(\Delta t+h^{2})\norm{\v_h}+Ch^{2} \norm{\v_h}_{H^1}.
\end{equation}
The Young's inequality yields \eqref{a4}.
\end{proof}
\subsection{Mathematical induction and  consequences}
We define the following error functions for $k = 0,1, \cdots, N$:
\[
\e_h^k = I_h\m(t_k)-\m_h^k, \quad \tilde{\e}_h^k = I_h\m(t_k)-\tilde{\m}_h^k,
\]
with $\e_h^0 = \tilde{\e}_h^0 = 0$.

We first present the fundamental relations between $\e_h^k$ and $\tilde{\e}_h^k$ for $k = 0,1, \cdots, N$ (see \cite{MR4377027, MR4883780}):
\begin{align}
  \norm{\e_h^k}_{L_h^p}\le C\norm{\tilde{\e}_h^k}_{L_h^p},\quad 
  \norm{\e_h^k}_{L^p}\le C\norm{\tilde{\e}_h^k}_{L^p},\ 1\le p\le +\infty.\label{e2}
\end{align}
As mentioned before, the above relations are not sufficient to derive the unconditional optimal error estimate. We thus need the following ideal relationship in the discrete norm.
\begin{lemma}\label{lemma4}
    For the scheme \eqref{a2}--\eqref{a8}, the auxiliary error
    \(\widetilde\e_h^{n+1}\) and the projected error \(\e_h^{n+1}\) satisfy
    \begin{align*}
      \norm{ \e_h^{n+1}}_h\le \norm{ \tilde\e_h^{n+1}}_h.
    \end{align*}
\end{lemma}
  \begin{proof}Since $I_h(\m(t_{n+1}))(\z) = \m(\z, t_{n+1})$ for all $\z\in \mathcal{Q}$, we only need to prove
  \[
  |\m(\z, t_{n+1}) - \m_h^{n+1}(\z)| \le |\m(\z, t_{n+1}) - \tilde{\m}_h^{n+1}(\z)|, \quad \forall\z\in \mathcal{Q}.
  \]
  In fact, using \eqref{c7}, we have
  \begin{align*}
    &|\m(\z, t_{n+1}) - \m_h^{n+1}(\z)|^2 - |\m(\z, t_{n+1}) - \tilde{\m}_h^{n+1}(\z)|^2 \\
    = &2\l \m(\z, t_{n+1})\cdot \tilde{\m}_h^{n+1}(\z) - \m(\z, t_{n+1})\cdot \m_h^{n+1}(\z)\r+1 - |\tilde{\m}_h^{n+1}(\z)|^2\\
    = &2 \m(\z, t_{n+1})\cdot \tilde{\m}_h^{n+1}(\z)\l1-\frac{1}{|\tilde{\m}_h^{n+1}(\z)|} \r+1-|\tilde{\m}_h^{n+1}(\z)|^2\\
    \le&2\l  |\tilde{\m}_h^{n+1}(\z)|-1\r-\l 1 + |\tilde{\m}_h^{n+1}(\z)|\r\l |\tilde{\m}_h^{n+1}(\z)|-1 \r \\
    =&\l  |\tilde{\m}_h^{n+1}(\z)|-1\r\l 1 - |\tilde{\m}_h^{n+1}(\z)|\r\le 0.
  \end{align*}
\end{proof}

We shall establish the following estimates for $k=0, 1, \cdots, N$ by using mathematical induction:
\begin{align}
  \max_{0\le l\le k}\norm{ \tilde{\e}_h^{l}}^2+\Delta t\sum_{l=1}^{k}\norm{\grad \tilde\e_h^{l}}_h^2\le& C(\Delta t^2+h^{4})\label{e0},\\
  \norm{\tilde\Delta_h \tilde\e_h^{k}}_h\le& M\label{e1}.
\end{align}
where $M$ is a fixed constant independent of $h$, $\Delta t$, and the time level $k$.

Now we present some consequences based on the induction assumptions. Using \eqref{e0}, \eqref{e1}, Lemma~\ref{lemma3}, and the discrete
Sobolev inequality \eqref{b5}, we obtain
\begin{align}\label{b1}
  \norm{\tilde{\e}_h^{k}}_{L^\infty}
  \le C(M)(\Delta t+h^2)^{\frac{1}{4}}.
\end{align}
Inequality \eqref{b1} guarantees that $\tilde\m_{h}^k$ is in a sufficiently small neighborhood of $\mathbb{S}^2$.
Based on \eqref{b1}, the following result presented in Lemma 3.5 of \cite{MR4883780} and Lemma 3.8 of \cite{MR4377027} connects the two errors in the $W^{1,p}$ norm:
\begin{align}
  &\norm{\e_h^k}_{W^{1,p}}\le C\norm{\tilde{\e}_h^k}_{W^{1,p}}+Ch,\quad  2\le p\le+\infty.\label{c2}
\end{align}
In particular, by combining  \eqref{e2}, \eqref{e0} and \eqref{c2}, we have
\begin{align}\label{t7}
  &\max_{1\le l\le k} \norm{\e_h^{{l}}}^2\le  C(\Delta t^2+h^{4}),\quad \Delta t\sum_{l=1}^{k}\norm{\e_h^l}_{H^1}^2\le C(\Delta t^2+h^{2}),
\end{align}
Together with the triangle inequality and interpolation error estimates, this means that completing the induction also completes the proof of Theorem~\ref{th1}.

Additionally, by the definition of the discrete operator $\Delta_h$, we obtain 
\begin{equation*}
\norm{\grad \tilde{\e}_h^{k}}^2=(\grad \tilde{\e}_h^{k},\grad \tilde{\e}_h^{k})=-(\Delta_h \tilde{\e}_h^{k},\tilde{\e}_h^{k})\le \norm{\Delta_h \tilde{\e}_h^{k}}\norm{\tilde{\e}_h^{k}}.
\end{equation*}
Thus, employing the Lebesgue interpolation inequality \cite[Theorem 1.5]{robinson2016three} and using \eqref{b6}, we derive
\begin{align}
  \norm{\grad \tilde{\e}_h^{k}}_{L^4}\le&\norm{\grad \tilde{\e}_h^{k}}^{\frac{1}{4}}\norm{\grad \tilde{\e}_h^{k}}_{L^6}^{\frac{3}{4}}\le C\norm{\tilde{\e}_h^{k}}^{\frac{1}{8}}\norm{\Delta_h \tilde{\e}_h^{k}}^{\frac{1}{8}+\frac{3}{4}}.\label{d4}
\end{align}
Then \eqref{e0}--\eqref{d4} and Lemma \ref{lemma3} imply that 
\begin{align}\label{b2}
  \norm{\grad {\e}_h^{k}}_{L^4}\le\norm{\grad \tilde{\e}_h^{k}}_{L^4}+Ch\le C\norm{\tilde{\e}_h^{k}}^{\frac{1}{8}}\norm{\Delta_h \tilde{\e}_h^{k}}^{\frac{7}{8}}+Ch\le C(M)(\Delta t+h^{2})^{\frac{1}{8}}.
\end{align}
Taking sufficiently small $h$ and $\Delta t$ in \eqref{b1} and \eqref{b2}, we have
\begin{align}\label{t6}
  \norm{\tilde{\e}_h^{k}}_{L^\infty}
  \le 1, \quad \norm{\grad {\e}_h^{k}}_{L^4}\le 1.
\end{align}
By inequality \eqref{e2}, \eqref{t6} and \eqref{p7}, together with the triangle inequality and the regularity assumptions, we have the following bounds independent of $M$
\begin{equation}\label{n8}
  \norm{\e_{h}^k}_{L^\infty_h}\le C, \ \norm{\m_{h}^k}_{L^\infty_h}\le C,\  \norm{\grad \e_{h}^k}_{L^4_h}\le C,\  \norm{\grad \m_{h}^k}_{L^4_h}\le C.
\end{equation}
The inequalities \eqref{t6} and \eqref{n8} are used in estimating the nonlinear terms.

Under the above induction assumptions, we also have the following lemmas that will be used in the error estimates and the completion of the induction argument. The proofs are provided in Appendices \ref{u3} and \ref{u4}, respectively.
\begin{lemma}\label{lemma10}
  Under the induction assumptions \eqref{e0}--\eqref{e1}, the following estimate holds for sufficiently small $\Delta t$ and $h$,
  \begin{align*}
    \norm{\grad\e_h^k}_h^2 
    \le&C_6\l \norm{\tilde\e_h^k}_h^2 + \norm{\grad\tilde\e_h^k}_h^2\r.
  \end{align*}
\end{lemma}
\begin{lemma}\label{lemma8}
  Under the assumptions \eqref{e0}--\eqref{e1}, there exists a positive constant $C_0(M)$, depending on \(M\) but independent of \(h\), \(\Delta t\), and \(k\), such that
\begin{equation}\label{p6}
  \|\tilde\Delta_h\m_h^k\|_h\le C_0(M).
\end{equation}
\end{lemma}

For \(k=0\), \eqref{e0}--\eqref{e1} hold trivially. Assume that
\eqref{e0}--\eqref{e1} hold for all time levels up to \(k=n\). We shall
prove that they also hold at \(k=n+1\) in the following subsections.
\subsection{ \texorpdfstring{Proof of \eqref{e0} for $k = n+1$}{Proof of (4.14) for k = n+1}}
Subtracting \eqref{u2} from \eqref{a1}, we obtain the following error equation:
\begin{align}
  &\l\frac{\tilde{\e}_h^{n+1}-\e_h^{n}}{\Delta t},\v_h \r_h+\l \grad \tilde{\e}_h^{n+1} , \grad \v_h\r_h\notag\\
  =\ &\l \grad I_h\m(t_{n})\cdot\grad I_h\m(t_{n+1}) , I_h\m(t_{n})\cdot\v_h \r_h - \l \grad\m_h^n\cdot\grad \tilde{\m}_h^{n+1} ,\m_h^n\cdot\v_h\r_h\notag\\
  &-\l (I_h\m(t_{n+1})- I_h\m(t_{n}))\cdot\grad I_h\m(t_{n}), \grad(I_h\m(t_{n})\cdot\v_h)\r_h\notag\\
  &+\l(\tilde{\m}_h^{n+1} - \m_h^n)\cdot\grad\m_h^n,\grad(\m_h^n\cdot\v_h)\r_h+\Ep(\v_h)\notag\\
  &-\l \grad((\tilde{\m}_h^{n+1} - \m_h^n)\cdot\m_h^n ),\grad(\m_h^n\cdot\v_h)\r_h\notag\\
  &+\l\grad I_h(\m(t_{n+1}) - \m(t_n)),\grad I_h(I_h\m(t_{n})(I_h\m(t_{n})\cdot\v_h))\r_h\notag\\
  &-\l\grad I_h(\m(t_{n+1}) - \m(t_n)),\grad( I_h\m(t_{n})(I_h\m(t_{n})\cdot\v_h))\r_h\notag\\
  &-\l\grad (\tilde{\m}_h^{n+1} - \m_h^n),\grad( I_h(\m_h^n(\m_h^n\cdot\v_h)) - \m_h^n(\m_h^n\cdot\v_h))\r_h\notag\\
  \eqqcolon&\Ep_5(\v_h)+\Ep_6(\v_h)+\Ep_7(\v_h)+\Ep_8(\v_h)+\Ep(\v_h).\label{a3}
\end{align}

Substituting $\v_h = \tilde{\e}_{h}^{n+1}$ into \eqref{a3}, invoking Lemma \ref{lemma4} with $\e_h^n$, we obtain:
\begin{align}
  &\frac{\norm{\tilde{\e}_h^{n+1}}_h^2-\norm{\e_h^{n}}_h^2}{2\Delta t}+\norm{\grad \tilde{\e}_h^{n+1}}_h^2
  \le \sum_{i=5}^8\Ep_i(\tilde{\e}_{h}^{n+1})+\Ep(\tilde{\e}_h^{n+1}).\label{d1}
\end{align}
Using \eqref{n8}, \eqref{t1}, Lemma \ref{lemma10} and the Young inequality, we have
\begin{align*}
  &\Ep_5(\tilde{\e}_{h}^{n+1})\\
  \le&\l \grad \e_h^n\cdot\grad I_h\m(t_{n+1}) , I_h\m(t_{n})\cdot\tilde{\e}_h^{n+1} \r_h+\l \grad \m_h^n\cdot\grad \tilde{\e}_{h}^{n+1} , I_h\m(t_{n})\cdot\tilde{\e}_h^{n+1} \r_h \notag\\
  &+ \l \grad\m_h^n\cdot\grad (I_h\m(t_{n+1}) - \tilde\e_h^{n+1}) ,\e_h^n\cdot\tilde{\e}_h^{n+1}\r_h\notag\\
  \le&C\norm{\grad \e_h^n}_h\norm{\grad I_h\m(t_{n+1})}_{L_h^\infty}\norm{I_h\m(t_{n})}_{L^\infty_h}\norm{\tilde{\e}_h^{n+1}}_h\\
  &+\norm{\grad\m_h^n}_{L_h^4}\norm{\grad\tilde\e_h^{n+1}}_h\norm{I_h\m(t_{n})}_{L^\infty_h}\norm{\tilde\e_h^{n+1}}_{L_h^4}\\
  &+\norm{\grad\m_h^n}_{L_h^4}\l\norm{\grad I_h\m(t_{n+1})}_{L_h^\infty}\norm{\e_h^n}_h+\norm{\tilde\e_h^{n+1}}_h\norm{\e_h^n}_{L_h^\infty}\r\norm{\tilde\e_h^{n+1}}_{L_h^4}\\
  \le&C\l \norm{\tilde{\e}_{h}^{n+1}}_h^2+\norm{\e_{h}^{n}}_h^2+\norm{\tilde\e_h^{n+1}}_{L_h^4}^2\r+\epsilon \l\norm{\grad\tilde{\e}_{h}^{n+1}}_h^2+\frac{1}{C_6}\norm{\grad\e_h^n}_h^2\r\\
  \le&C\l \norm{\tilde{\e}_{h}^{n+1}}_h^2+\norm{\tilde{\e}_{h}^{n}}_h^2\r+2\epsilon \l\norm{\grad\tilde{\e}_{h}^{n+1}}_h^2+\norm{\grad\tilde\e_h^n}_h^2\r,
\end{align*}
where we used 
\begin{equation*}
\norm{\grad I_h\m(t_{n+1})}_{L_h^\infty}+\norm{ I_h\m(t_{n+1})}_{L_h^\infty}\le C\norm{\m(t_{n+1})}_{W^{1,\infty}}<C.
\end{equation*} 
A similar procedure yields
\begin{align*}
  &\Ep_6(\tilde{\e}_{h}^{n+1})\\
  =&\l (\e_h^n - \tilde\e_h^{n+1})\cdot\grad I_h\m(t_{n}), \grad( I_h\m(t_{n})\cdot\tilde{\e}_{h}^{n+1})\r_h\notag\\
  &+\l (I_h\m(t_n) -\e_h^n - I_h\m(t_{n+1}) + \tilde\e_h^{n+1})\cdot\grad \e_h^n, \grad( I_h\m(t_{n})\cdot\tilde{\e}_{h}^{n+1})\r_h\notag\\
  &+\l (\m_h^n - (I_h\m(t_{n+1}) - \tilde\e_h^{n+1}))\cdot\grad \m_h^n, \grad( \e_h^n\cdot\tilde{\e}_{h}^{n+1})\r_h\notag\\
  \le&C\norm{\e_h^n - \tilde\e_h^{n+1}}_h\norm{\grad I_h\m(t_n)}_{L_h^\infty}\l \norm{\tilde{\e}_{h}^{n+1}}_h + \norm{\grad\tilde{\e}_{h}^{n+1}}_h\r\\
  &+\norm{I_h(\m(t_n) - \m(t_{n+1}))}_{L_h^4}\norm{\grad \e_h^n}_{L_h^4} \l \norm{\tilde{\e}_{h}^{n+1}}_h + \norm{\grad\tilde{\e}_{h}^{n+1}}_h\r\\
  &+\l\norm{\e_h^n}_{L_h^4}+\norm{\tilde\e_h^{n+1}}_{L_h^4}\r\norm{\grad \e_h^n}_{L_h^4} \l \norm{\tilde{\e}_{h}^{n+1}}_h + \norm{\grad\tilde{\e}_{h}^{n+1}}_h\r\\
  &+\norm{\m_h^n - I_h\m(t_{n+1})}_{L_h^\infty}\norm{\grad\m_h^n}_{L_h^4}\l \norm{\grad\e_h^n}_h\norm{\tilde\e_h^{n+1}}_{L_h^4} + \norm{\e_h^n}_{L_h^4}\norm{\grad\tilde\e_h^{n+1}}_h\r\\
  &+\norm{\tilde\e_h^{n+1}}_{L_h^4}\norm{\grad\m_h^n}_{L_h^4}\l \norm{\grad \e_h^n}_{L^4_h}\norm{\tilde{\e}_{h}^{n+1}}_{L^4_h} + \norm{\e_h^n}_{L^\infty_h}\norm{\grad\tilde{\e}_{h}^{n+1}}_h\r\\
  \le&C\l\Delta t^2+\norm{\e_{h}^{n}}_h^2+\norm{\tilde{\e}_{h}^{n+1}}_h^2+\norm{\tilde{\e}_{h}^{n+1}}_{L_h^4}^2+\norm{\e_{h}^{n}}_{L_h^4}^2\r+\epsilon (\norm{\grad\tilde{\e}_{h}^{n+1}}_h^2+\frac{1}{C_6}\norm{\grad\e_h^n}_h^2)\\
  \le&C\l\Delta t^2+\norm{\tilde\e_{h}^{n}}_h^2+\norm{\tilde{\e}_{h}^{n+1}}_h^2\r+2\epsilon \l\norm{\grad\tilde{\e}_{h}^{n+1}}_h^2+\norm{\grad\tilde\e_h^n}_h^2\r.
\end{align*}

Next, since $(\tilde{\m}_h^{n+1}(\z) - \m_h^n(\z))\cdot\m_h^n(\z)=0$ holds for all $\z\in\mathcal{Q}$, on the element $K=\prod_{i=1}^d[x_{K,i}^-,x_{K,i}^+]$, we have
\begin{align*}
  0=&
  (\tilde{\m}_h^{n+1} - \m_h^n)(\z_{K,i,\alpha'}^+)\cdot \m_h^n(\z_{K,i,\alpha'}^+) - (\tilde{\m}_h^{n+1} - \m_h^n)(\z_{K,i,\alpha'}^-)\cdot \m_h^n(\z_{K,i,\alpha'}^-)\\
  =&h_{K,i}\l \mathfrak D_{K,i}^{\alpha'}(\tilde{\m}_h^{n+1} - \m_h^n)\cdot\m_h^n(\z_{K,i,\alpha'}^-) + (\tilde{\m}_h^{n+1} - \m_h^n)(\z_{K,i,\alpha'}^-)\cdot\mathfrak D_{K,i}^{\alpha'}\m_h^n\r\\
  &+ h_{K,i}^2\mathfrak D_{K,i}^{\alpha'}(\tilde{\m}_h^{n+1} - \m_h^n)\cdot\mathfrak D_{K,i}^{\alpha'}\m_h^n,
\end{align*}
where we have noted that
\begin{align*}
  (\tilde{\m}_h^{n+1} - \m_h^n)(\z_{K,i,\alpha'}^+)&=(\tilde{\m}_h^{n+1} - \m_h^n)(\z_{K,i,\alpha'}^-)+h_{K,i}\mathfrak D_{K,i}^{\alpha'}(\tilde{\m}_h^{n+1} - \m_h^n)\\
  \m_h^n(\z_{K,i,\alpha'}^+)&=\m_h^n(\z_{K,i,\alpha'}^-)+h_{K,i}\mathfrak D_{K,i}^{\alpha'}\m_h^n.
\end{align*}
Hence
\begin{align*}
 &\mathfrak D_{K,i}^{\alpha'}(\tilde{\m}_h^{n+1} - \m_h^n)\cdot\m_h^n(\z_{K,i,\alpha'}^-) + (\tilde{\m}_h^{n+1} - \m_h^n)(\z_{K,i,\alpha'}^-)\cdot\mathfrak D_{K,i}^{\alpha'}\m_h^n\\
 =&-h_{K,i}\mathfrak D_{K,i}^{\alpha'}(\tilde{\m}_h^{n+1} - \m_h^n)\cdot\mathfrak D_{K,i}^{\alpha'}\m_h^n,
\end{align*}
which means that 
\[
  \partial_i(((\tilde{\m}_h^{n+1} - \m_h^n)\cdot\m_h^n)|_K)
  (\z_{K,i,\alpha'}^-)
  =
  -h_{K,i}
  \mathfrak D_{K,i}^{\alpha'}(\tilde{\m}_h^{n+1} - \m_h^n)
  \cdot
  \mathfrak D_{K,i}^{\alpha'}\m_h^n.
\]
Similarly,
\[
  \partial_i(((\tilde{\m}_h^{n+1} - \m_h^n)\cdot\m_h^n)|_K)
  (\z_{K,i,\alpha'}^+)
  =
  h_{K,i}
  \mathfrak D_{K,i}^{\alpha'}(\tilde{\m}_h^{n+1} - \m_h^n)
  \cdot
  \mathfrak D_{K,i}^{\alpha'}\m_h^n.
\]
We thus have for every $\z\in\mathcal Q_K$,
\[
  \left|
  \partial_i\l((\tilde{\m}_h^{n+1} - \m_h^n)\cdot\m_h^n)|_K\r(\z)
  \right|
  \le
  h_{K,i}
  \left|
  \partial_i((\tilde{\m}_h^{n+1} - \m_h^n)|_K)(\z)
  \right|
  \left|
  \partial_i(\m_h^n|_K)(\z)
  \right|,
\]
which along with the H\"older inequality, Lemma \ref{lemma10} and \eqref{e7}, means that
\begin{align*}
  &\Ep_7(\tilde{\e}_{h}^{n+1})\\
  =&\l \grad ((\tilde{\m}_h^{n+1} - \m_h^n)\cdot\m_h^n),\grad(\m_h^n\cdot\tilde\e_h^{n+1})\r_h\\
  =&\sum_{K\in\mathscr{T}_h}\sum_{\z\in\mathcal{Q}_K}\omega_{K,\z}\sum_{i=1}^d\partial_i\l((\tilde{\m}_h^{n+1} - \m_h^n)\cdot\m_h^n)|_K\r(\z)\cdot\partial_i\l(\m_h^n\cdot\tilde\e_h^{n+1})|_K\r(\z)\\
  \le&Ch\norm{\grad(\tilde{\m}_h^{n+1} - \m_h^n)}_h\norm{\grad\m_h^n}_{L_h^4}\norm{\grad(\m_h^n\cdot\tilde\e_h^{n+1})}_{L_h^4}\\
  \le&C\norm{\grad(\tilde{\m}_h^{n+1} - \m_h^n)}_h\norm{\grad\m_h^n}_{L_h^4}\norm{\m_h^n}_{L_h^\infty}\norm{\tilde\e_h^{n+1}}_{L^4_h}\\
  \le&C\norm{\grad(I_h\m(t_{n+1}) - \tilde\e_h^{n+1} - I_h\m(t_n) + \e_h^n)}_h\norm{\tilde\e_h^{n+1}}_{L^4_h}\\
  \le&C(\Delta t^2 + \norm{\tilde\e_h^{n}}_h^2 + \norm{\tilde\e_h^{n+1}}_h^2) + \epsilon\l \norm{\grad\tilde\e_h^{n+1}}_h^2 + \norm{\grad\tilde\e_h^n}_h^2\r.
\end{align*}
By the same argument as that used to derive \eqref{y7}, we obtain
\begin{align}\label{u8}
  \norm{\widetilde\Delta_h I_h\m(t_{n+1})}_h
  \le C\norm{\m(t_{n+1})}_{H^2}
  \le C .
\end{align}
Using \eqref{y7}, \eqref{u7} with $\F_h^1 = \m_h^n(\m_h^n\cdot\tilde\e_h^{n+1})$ and \eqref{u8}, we have
\begin{align*}
  &\Ep_8(\tilde\e_h^{n+1})\\
  =&\l\grad I_h(\m(t_{n+1}) - \m(t_n)),\grad\l I_h(I_h\m(t_{n})(I_h\m(t_{n})\cdot\tilde\e_h^{n+1})) \r\r_h\\
  &-\l\grad I_h(\m(t_{n+1}) - \m(t_n)),\grad\l I_h(\m_h^n(\m_h^n\cdot\tilde\e_h^{n+1}))\r\r_h\\
  &-\l\grad I_h(\m(t_{n+1}) - \m(t_n)),\grad\l I_h\m(t_{n})(I_h\m(t_{n})\cdot\tilde\e_h^{n+1}) - \m_h^n(\m_h^n\cdot\tilde\e_h^{n+1}) \r\r_h\\
  &+\l\grad (\tilde\e_h^{n+1} - \e_h^n),\grad\l I_h(\m_h^n(\m_h^n\cdot\tilde\e_h^{n+1})) - \m_h^n(\m_h^n\cdot\tilde\e_h^{n+1}) \r\r_h\\
  =&-\l\tilde\Delta_h I_h(\m(t_{n+1}) - \m(t_n)),  \e_h^n(I_h\m(t_n)\cdot\tilde\e_h^{n+1}) + \m_h^n(\e_h^n\cdot\tilde\e_h^{n+1})\r_h\\
  &-\l\grad I_h(\m(t_{n+1})-\m(t_n)),\grad \l \e_h^n(I_h\m(t_n)\cdot\tilde\e_h^{n+1})+\m_h^n(\e_h^n\cdot\tilde\e_h^{n+1}) \r\r_h\\
  &+\sum_{K\in\mathscr{T}_h}\sum_{i=1}^d\l \partial_i(\tilde\e_h^{n+1} - \e_h^n), \partial_i\l I_h(\m_h^n(\m_h^n\cdot\tilde\e_h^{n+1})) - \m_h^n(\m_h^n\cdot\tilde\e_h^{n+1}) \r\r_{K,h}\\
  \le&C\norm{\tilde\Delta_h I_h(\m(t_{n+1}) - \m(t_n))}_h\norm{\e_h^n}_{L^\infty_h}(\norm{I_h\m(t_n)}_{L_h^\infty}+\norm{\m_h^n}_{L_h^\infty})\norm{\tilde\e_h^{n+1}}_h\\
  &+C\norm{\grad I_h(\m(t_{n+1})-\m(t_n))}_{L_h^4}\l\norm{\grad\e_h^n}_{L_h^4}\norm{\tilde\e_h^{n+1}}_h + \norm{\e_h^n}_{L_h^\infty}\norm{\tilde\e_h^{n+1}}_{H^1_h} \r\\
  &+\sum_{K\in\mathscr{T}_h}\sum_{i=1}^d\l |\partial_i(\tilde\e_h^{n+1} - \e_h^n)|, \frac{h_{K,i}}{2}|\partial_{ii}(\m_h^n(\m_h^n\cdot\tilde\e_h^{n+1}))|\r_{K,h}\\
  &+\sum_{K\in\mathscr{T}_h}\sum_{i=1}^d\l |\partial_i(\tilde\e_h^{n+1} - \e_h^n)|, \frac{h_{K,i}^2}{6}|\partial_{iii}(\m_h^n(\m_h^n\cdot\tilde\e_h^{n+1}))|\r_{K,h}\\
  \le&C\l\Delta t^2 + \norm{\tilde\e_h^{n+1}}_h^2 \r + \epsilon \norm{\grad \tilde\e_h^{n+1}}_h^2 \\
  &+C\sum_{K\in\mathscr{T}_h}\norm{\grad(\tilde\e_h^{n+1} - \e_h^n)}_{L_h^2(K)}\norm{\grad \m_h^n}_{L^4_h(K)}\norm{\m_h^n}_{L^\infty_h(K)}\norm{\tilde\e_h^{n+1}}_{L^4_h(K)}\\
  \le&C\l\Delta t^2+\norm{\tilde\e_h^{n+1}}_h^2 + \norm{\tilde\e_h^{n}}_h^2 \r + 2\epsilon\l \norm{\grad \tilde\e_h^{n+1}}_h^2 + \norm{\grad\tilde\e_h^n}_h^2\r.
\end{align*}

Combining the estimates of $\Ep_i(\tilde\e_h^{n+1})$ for $i=5, 6, 7, 8$ and $\Ep(\tilde\e_h^{n+1})$ in \eqref{d1}, multiplying by $2\Delta t$, summing over time levels, and choosing $\epsilon$ small enough, we obtain
\begin{align*}
  &\norm{ \tilde{\e}_h^{n+1}}_h^2 +\Delta t\sum_{l=0}^n\norm{\grad \tilde\e_h^{l+1}}_h^2
  \le C\Delta t\sum_{l=0}^n \norm{\tilde{\e}_h^{l+1}}_h^2+C(\Delta t^2+h^4).
\end{align*} 
The discrete Gronwall inequality and \eqref{b3} lead to
\begin{align}\label{a5}
  &\max_{0\le l\le n+1}\norm{ \tilde{\e}_h^{l}}^2+\Delta t\sum_{l=1}^{n+1}\norm{\grad \tilde\e_h^{l}}_h^2\le C(\Delta t^2+h^{4}).
\end{align}
\subsection{ \texorpdfstring{Proof of \eqref{e1} for $k = n+1$}{Proof of (4.15) for k = n+1}}
From the error equation \eqref{a3}
we have
\begin{align}\label{y1}
  &\l \tilde\Delta_h \tilde{\e}_h^{n+1} ,  \v_h\r_h = \l\frac{\tilde{\e}_h^{n+1}-\e_h^{n}}{\Delta t},\v_h \r_h-\sum_{i=5}^8\Ep_i(\v_h)-\Ep(\v_h).
\end{align}
Using \eqref{a5}, the estimation \eqref{c5} for $\Ep(\v_h)$, the inverse inequality, \eqref{n8} and equivalence \eqref{p7}, we have for sufficiently small $\Delta t$ and $h$
\begin{align*}
  &\l\frac{\tilde{\e}_h^{n+1}-\e_h^{n}}{\Delta t},\v_h \r_h - \Ep(\v_h) - \Ep_5(\v_h)\\
  \le&C\Delta t^{-1}(\Delta t+h^2)\norm{\v_h}_h+C(\Delta t+h)\norm{\v_h}_h\\
  &+C\norm{\grad I_h\m(t_n)}_{L_h^4}\norm{\grad I_h\m(t_n)}_{L_h^4}\norm{I_h\m(t_n)}_{L_h^\infty}\norm{\v_h}_h\notag\\
  &+C\norm{\grad \m_h^n}_{L_h^4}\norm{\grad \tilde\m_h^{n+1}}_{L_h^4}\norm{\m_h^n}_{L_h^\infty}\norm{\v_h}_h\\
  \le& C(1+\Delta t^{-1}h^2+\norm{\grad\tilde\e_h^{n+1}}_{L^4})\norm{\v_h}_h.
\end{align*}
Next, we split $\Ep_6(\v_h)$ as follows:
\begin{align*}
  &\Ep_6(\v_h)\\
  =&-\l\grad I_h\m(t_{n}), \grad\l (I_h\m(t_{n+1}) - I_h\m(t_{n}))( I_h\m(t_{n})\cdot\v_h)\r\r_h\\
  &+\l\grad I_h\m(t_{n}), ( I_h\m(t_{n})\cdot\v_h)\grad(I_h\m(t_{n+1}) - I_h\m(t_{n}))\r_h\\
  &+\l\grad\m_h^n,\grad\l(\tilde{\m}_h^{n+1} - \m_h^n)(\m_h^n\cdot\v_h)\r\r_h-\l\grad\m_h^n,(\m_h^n\cdot\v_h)\grad(\tilde{\m}_h^{n+1} - \m_h^n)\r_h\\
  \eqqcolon&\Ep_{6,1}(\v_h)+\Ep_{6,2}(\v_h)+\Ep_{6,3}(\v_h)+\Ep_{6,4}(\v_h).
\end{align*}
First, using \eqref{r7}, the inverse inequality and the interpolation error estimate, we have
\begin{align*}
  &\Ep_{6,1}(\v_h)\\
  =&-\Ep_{\rm quad}\l\grad I_h\m(t_{n}), \grad\l (I_h\m(t_{n+1}) - I_h\m(t_{n}))( I_h\m(t_{n})\cdot\v_h)\r\r\\
  &-\l\grad (I_h\m(t_{n}) - \m(t_n)), \grad\l (I_h\m(t_{n+1}) - I_h\m(t_{n}))( I_h\m(t_{n})\cdot\v_h)\r\r\\
  &+\l\Delta\m(t_n),  (I_h\m(t_{n+1}) - I_h\m(t_{n}))( I_h\m(t_{n})\cdot\v_h)\r\\
  \le&Ch^2\sum_{K\in\mathscr{T}_h}\sum_{i=1}^d\norm{\partial_i^2\l \grad I_h\m(t_{n})\cdot\grad\l (I_h\m(t_{n+1}) - I_h\m(t_{n}))( I_h\m(t_{n})\cdot\v_h)\r \r}\\
  &+Ch\norm{\m(t_n)}_{H^2}\norm{\grad\l(I_h\m(t_{n+1}) - I_h\m(t_{n}))( I_h\m(t_{n})\cdot\v_h)\r}\\
  &+C\norm{\m(t_n)}_{H^2}\norm{(I_h\m(t_{n+1}) - I_h\m(t_{n}))( I_h\m(t_{n})\cdot\v_h)}\\
  \le&C\l \norm{I_h\m(t_n)}_{H^2}\norm{I_h\m(t_{n+1}) - I_h\m(t_n)}_{W^{1,\infty}}\norm{I_h\m(t_{n})}_{W^{1,\infty}}\r\norm{\v_h}_h\\
  &+C\norm{\m(t_n)}_{H^2}\norm{I_h\m(t_{n+1}) - I_h\m(t_{n})}_{L^\infty}\norm{I_h\m(t_{n})}_{L^\infty}\norm{\v_h}_h\\
  \le&C\norm{\v_h}_h.
\end{align*}
Next, using \eqref{u7} with $\F^1_h = (\tilde{\m}_h^{n+1} - \m_h^n)(\m_h^n\cdot\v_h)$, \eqref{b1}, \eqref{n8} and the inverse inequality, we have
\begin{align*}
  &\Ep_{6,3}(\v_h) \\
  =&\l\grad\m_h^n,\grad\l(\tilde{\m}_h^{n+1} - \m_h^n)(\m_h^n\cdot\v_h) - I_h((\tilde{\m}_h^{n+1} - \m_h^n)(\m_h^n\cdot\v_h))\r\r_h\\
  &-\l\tilde\Delta_h\m_h^n,(\tilde{\m}_h^{n+1} - \m_h^n)(\m_h^n\cdot\v_h)\r_h\\
  \le&\sum_{K\in\mathscr{T}_h}\sum_{i=1}^d\l |\partial_i \m_h^n|, \frac{h_{K,i}}{2}|\partial_{ii}((\tilde{\m}_h^{n+1} - \m_h^n)(\m_h^n\cdot\v_h))|\r_{K,h}\\
  &+\sum_{K\in\mathscr{T}_h}\sum_{i=1}^d\l |\partial_i \m_h^n|, \frac{h_{K,i}^2}{6}|\partial_{iii}(  (\tilde{\m}_h^{n+1} - \m_h^n)(\m_h^n\cdot\v_h))|\r_{K,h}\\
  &+\norm{\tilde\Delta_h\m_h^n}_h\norm{I_h\m(t_{n+1}) - \tilde\e_h^{n+1} - I_h\m(t_n) + \e_h^n}_{L_h^\infty}\norm{\m_h^n}_{L_h^\infty}\norm{\v_h}_h\\
  \le&\sum_{K\in\mathscr{T}_h}\sum_{i=1}^d\norm{\grad\m_h^n}_{L_h^4(K)}\l\norm{\grad(\tilde\m_h^{n+1} - \m_h^{n})}_{L_h^4(K)}\norm{\m_h^n}_{L_h^\infty(K)}\r\norm{\v_h}_h\\
  &+C\l\Delta t+\norm{\tilde\e_h^{n+1}}_{L_h^\infty}+\norm{\e_h^{n}}_{L_h^\infty}\r\norm{\tilde\Delta_h\m_h^n}_h\norm{\v_h}_h\\
  \le&C\l1+\norm{\grad\tilde\e_h^{n+1}}_{L^4}\r\norm{\v_h}_h+C\l\norm{\tilde\e_h^{n+1}}_{L_h^\infty}+(\Delta t+h^2)^\frac{1}{4}\r\norm{\tilde\Delta_h\m_h^n}_h\norm{\v_h}_h.
\end{align*}
Furthermore, using the H\"older inequality and \eqref{n8} , it is easy to see that
\begin{align*}
  \Ep_{6,2}(\v_h)+\Ep_{6,4}(\v_h)\le C(1+\norm{\grad\tilde\e_h^{n+1}}_{L^4})\norm{\v_h}_h.
\end{align*}
As a result, we have
\begin{align*}
  \Ep_6(\v_h)
  \le C(\norm{\tilde\e_h^{n+1}}_{L_h^\infty}+(\Delta t+h^2)^\frac{1}{4})\norm{\tilde\Delta_h\m_h^n}_h\norm{\v_h}_h+C(1+\norm{\grad\tilde\e_h^{n+1}}_{L^4})\norm{\v_h}_h.
\end{align*}
Similarly to the estimate of $\Ep_7(\tilde{\e}_{h}^{n+1})$, using the equivalence  \eqref{p7} together with the inverse inequality, we have
\begin{align*}
  \Ep_7(\v_h)
  \le&Ch\norm{\grad(\tilde{\m}_h^{n+1} - \m_h^n)}_{L_h^4}\norm{\grad\m_h^n}_{L_h^4}\l\norm{\grad\m_h^n}_{L_h^\infty}\norm{\v_h}_h+\norm{\m_h^n}_{L^\infty_h}\norm{\grad \v_h}_h\r\\
  \le&C\norm{\grad(I_h\m(t_{n+1}) - \tilde\e_h^{n+1} - I_h\m(t_n) + \e_h^n)}_{L^4}\norm{\m_h^n}_{L_h^\infty}\norm{\v_h}_h\\
  \le&C(1+\Delta t+\norm{\grad\tilde\e_h^{n+1}}_{L^4})\norm{\v_h}_h  \le C(1+\norm{\grad\tilde\e_h^{n+1}}_{L^4})\norm{\v_h}_h.
\end{align*}
Finally, using \eqref{u7} with $\F^1_h = \m_h^n(\m_h^n\cdot\v_h)$ again and the discrete inverse inequality, we have
\begin{align*}
  &\Ep_8(\v_h)\\
  \le&\sum_{K\in\mathscr{T}_h}\sum_{i=1}^d\l \l |\partial_iI_h(\m(t_{n+1}) - \m(t_n))|, \frac{h_{K,i}}{2}|\partial_{ii}(I_h\m(t_{n})(I_h\m(t_{n})\cdot\v_h))|\r_{K,h}\\
  &+\l |\partial_iI_h(\m(t_{n+1}) - \m(t_n))|,\frac{h_{K,i}^2}{6} |\partial_{iii}(I_h\m(t_{n})(I_h\m(t_{n})\cdot\v_h))|\r_{K,h}\\
  &+\l |\partial_i(\tilde\m_h^{n+1} - \m_h^n)|, \frac{h_{K,i}}{2}|\partial_{ii}(\m_h^n(\m_h^n\cdot\v_h))|+\frac{h_{K,i}^2}{6} |\partial_{iii}(\m_h^n(\m_h^n\cdot\v_h))|\r_{K,h}\r\\
  \le&C\l \norm{\grad I_h(\m(t_{n+1}) - \m(t_n))}_{L^4_h}\norm{\grad I_h\m(t_{n})}_{L^4_h}\norm{I_h\m(t_{n})}_{L^\infty_h(K)}\norm{\v_h}_h\\
  &+\norm{\grad(\tilde\m_h^{n+1} - \m_h^n)}_{L^4_h}\norm{\grad\m_h^n}_{L^4_h}\norm{\m_h^n}_{L^\infty_h}\norm{\v_h}_h\r\\
  \le&C(1+\norm{\grad\tilde\e_h^{n+1}}_{L^4})\norm{\v_h}_h.
\end{align*}

Combining the above estimations in \eqref{y1}, using \eqref{d4}, \eqref{b5} and \eqref{p2}, we derive
\begin{align*}
  &\norm{\tilde\Delta_h \tilde{\e}_h^{n+1}}_h
  = \sup_{\v_h\in\S_h^1 ,\v_h\neq 0}\frac{\left|\l \tilde\Delta_h \tilde{\e}_h^{n+1}, \v_h\r_h\right|}{\norm{\v_h}_h}\\
  \le& C(1+\Delta t^{-1}h^2+\norm{\grad \tilde{\e}_h^{n+1}}_{L^4})+C(\norm{\tilde\e_h^{n+1}}_{L^\infty}+(\Delta t+h^2)^\frac{1}{4})\norm{\tilde\Delta_h\m_h^n}_h\\
  \le&C\l 1+\Delta t^{-1}h^2+\norm{\tilde{\e}_h^{n+1}}^{\frac{1}{8}}\norm{\Delta_h \tilde{\e}_h^{n+1}}^{\frac{7}{8}}\r+C(\Delta t+h^2)^\frac{1}{4}\norm{\tilde\Delta_h\m_h^n}_h\\
  &+C\norm{\tilde\e_h^{n+1}}^{1-\frac{d}{4}}(\norm{\tilde\e_h^{n+1}}+\norm{\Delta_h\tilde\e_h^{n+1}})^{\frac{d}{4}}\norm{\tilde\Delta_h\m_h^n}_h\\
  \le&C\l 1+\Delta t^{-1}h^2+\norm{\tilde{\e}_h^{n+1}}\r+C(\Delta t+h^2)^\frac{1}{4}\norm{\tilde\Delta_h\m_h^n}_h\\
  &+C\norm{\tilde\e_h^{n+1}}\norm{\tilde\Delta_h\m_h^n}_h^{\frac{4}{4-d}}+\frac{\epsilon}{C_5}\norm{\tilde\Delta_h\tilde\e_h^{n+1}}.
\end{align*}
By the \eqref{p6} and note that \(d\le3\), selecting sufficiently small $\Delta t$, $h$ such that 
\begin{align*}
(\Delta t+h^2)^\frac{1}{4}\norm{\tilde\Delta_h\m_h^n}_h\le& (\Delta t+h^2)^\frac{1}{4}C_0(M)\le 1,
\end{align*}
then using \eqref{a5}, for sufficiently small $\epsilon$ and a positive constant $C_7>0$ independent of $M$, we have
\begin{align}\label{h5}
  \norm{\tilde\Delta_h \tilde{\e}_h^{n+1}}_h
  \le C_7(1+\Delta t^{-1}h^2).
\end{align}
On the other hand, by the definition of $\tilde\Delta_h$, we have
\begin{align*}
  \norm{\tilde\Delta_h\tilde{\e}_h^{n+1}}_h^2= -\l \grad\tilde{\e}_h^{n+1}, \grad\tilde\Delta_h\tilde{\e}_h^{n+1} \r_h\le C\norm{\grad\tilde{\e}_h^{n+1}}_h\norm{\grad\tilde\Delta_h\tilde{\e}_h^{n+1}}_h,
\end{align*}
which combines with the discrete inverse inequality and \eqref{a5}, we have 
\begin{align}\label{h6}
\norm{\tilde\Delta_h\tilde{\e}_h^{n+1}}_h\le Ch^{-2}\norm{\tilde{\e}_h^{n+1}}_h\le C_8(1+\Delta th^{-2}),
\end{align}
where $C_8$ independent of $M$.

Now, letting $M \coloneqq 2\max\{C_7, C_8\}$, combining \eqref{h5} and \eqref{h6}, we have
\[
\norm{\tilde\Delta_h\tilde{\e}_h^{n+1}}_h\le\min\{C_7(1+\Delta t^{-1}h^2), C_8(1+\Delta th^{-2})\}\le M.
\]

Thus the induction is closed at the level \(k=n+1\). Consequently,
\eqref{t7} also holds at \(k=n+1\), and the proof of
Theorem~\ref{th1} is complete.
\section{Numerical results}\label{sec5} In this section, we present numerical results concerning the convergence rates and energy stability of the proposed scheme, as well as potential finite-time blow-up, based on the open-source finite element library FEniCSx \cite{BarattaEtal2023}.
\subsection{Convergence tests}
\begin{table}
		\tabcolsep 1.8mm
         \caption{Temporal discretization errors and rates}
		\label{table6}
		\centering
		\begin{tabular}{ccccc}
			\toprule
			$\Delta t$ & $\max\limits_{1\le k\le N} \norm{\m_{h, \Delta t}^{k}-\m_{h, \Delta t/2}^{2k}}$ & Rate  & $\sqrt{\Delta t\sum\limits_{k=1}^{N}\norm{\m_{h, \Delta t}^{k}-\m_{h, \Delta t/2}^{2k}}_{H^1}^2}$ &Rate\\
			\midrule
      \multicolumn{5}{l}{two-dimensional}\\
			\midrule
			T/32 & $1.817e$-$03$ & $-$ & $2.790e$-$03$ & $-$\\
			
			T/64 & $9.357e$-$04$ & 0.96 & $1.445e$-$03$ & 0.95 \\
			
			T/128 & $4.771e$-$04$ & 0.97 & $7.399e$-$04$& 0.97\\

			T/256 & $2.405e$-$04$ & 0.99 & $3.740e$-$04$& 0.98\\
			\midrule
      \multicolumn{5}{l}{three-dimensional}\\
			\midrule
			T/16 & $3.711e$-$03$ & $-$ & $5.526e$-$03$ & $-$\\
			
			T/32 & $1.934e$-$03$ & 0.94 & $2.883e$-$03$ & 0.94 \\
			
			T/64 & $9.897e$-$04$ & 0.97 & $1.477e$-$03$& 0.96\\

			T/128 & $5.009e$-$04$ & 0.98 & $7.485e$-$04$& 0.98\\
			\bottomrule
		\end{tabular}
	\end{table}
We consider the problem on the $d$-dimensional domain $\Omega = [1/2, 3/2]^d$, $d=2, 3$  and let $T = 0.5$. The initial value of the solution is chosen to be 
\[
\m^0 = \frac{1}{S}[ \tilde m_1^0,\tilde m_2^0,\tilde m_3^0 ]^\top,
\]
where $S(\x) = \sqrt{\tilde m_1^0(\x)^2+\tilde m_2^0(\x)^2+\tilde m_3^0(\x)^2}$ and
\begin{align*}
  &\tilde m_1^0(x, y) = \sin(\pi  x)  \cos(2\pi y) + 1,\quad
  \tilde m_2^0(x, y) = \cos(2\pi x) \cos(2\pi y) + 2,\\
  &\tilde m_3^0(x, y) = \sin(\pi y)
\end{align*}
for $d = 2$ and 
\begin{align*}
  &\tilde m_1^0(x, y, z) = \sin(\pi  x)  \cos(2\pi y) + \cos^2(\pi z) + 1,\quad
  \tilde m_2^0(x, y, z) = \cos(2\pi x) \cos(2\pi y) + 2,\\
  &\tilde m_3^0(x, y, z) = \sin(\pi y) \cos^2(\pi z) + \sin(\pi x).
\end{align*}
for $d=3$. One can verify that the initial data satisfies the pointwise constraint and the boundary condition. Fix $h = 1/24$ for computing the temporal errors and rates with different $\Delta t$, and fix $\Delta t = T/10$ for computing the spatial errors and rates for different $h$. The temporal and spatial errors and convergence rates in two and three dimensions are presented in Tables \ref{table6} and \ref{table8}, respectively. We can observe that the convergence rates, both temporal and spatial, are consistent with the theoretical result in Theorem \ref{th1}.
	\begin{table}
		\tabcolsep 2mm
         \caption{Spatial discretization errors and rates}
		\label{table8}
		\centering
		\begin{tabular}{ccccc}
			\toprule
			$h$ & $\max\limits_{1\le k\le N} \norm{\m_{h, \Delta t}^{k}-\m_{h/2, \Delta t}^{k}}$ & Rate  & $\sqrt{\Delta t\sum\limits_{k=1}^{N}\norm{\m_{h, \Delta t}^{k}-\m_{h/2, \Delta t}^{k}}_1^2}$ &Rate\\
			\midrule
      \multicolumn{5}{l}{two-dimensional}\\
			\midrule
			1/16 & $2.635e$-$03$ & $-$ & $1.724e$-$02$ & $-$\\
			
			1/32 & $6.606e$-$04$ & 2.00 &$8.627e$-$03$ & 1.00 \\
			
			1/64 & $1.653e$-$04$ & 2.00 &$4.314e$-$03$& 1.00\\

			1/128 & $4.132e$-$05$ & 2.00 & $2.157e$-$03$& 1.00\\
			\midrule
      \multicolumn{5}{l}{three-dimensional}\\
			\midrule
			1/6 & $1.903e$-$02$ & $-$ & $4.799e$-$02$ & $-$\\
			
			1/12 & $4.920e$-$03$ & 1.95 &$2.411e$-$02$ & 0.99 \\
			
			1/24 & $1.239e$-$03$ & 1.99 & $1.207e$-$02$& 1.00\\

			1/48 & $3.102e$-$04$ & 2.00 & $6.035e$-$03$& 1.00\\
			\bottomrule
		\end{tabular}
	\end{table}
\subsection{Energy stability and possible finite-time blow-up}
Let $\Omega = (-1, 1)^2$ and $T = 10$. The initial value is defined by
\begin{equation*}
\m^0(\x) = \l \frac{\x}{|\x|}\sin\phi(|\x|), \cos\phi(|\x|) \r,\quad \text{ where }\phi(|\x|)\coloneqq \left\{
\begin{aligned}
  &3\pi|\x|^2/2\quad \text{ for }|\x|\le 1,\\
  &3\pi/2\quad\ \ \quad\text{ for }|\x|\ge 1.
\end{aligned}
\right.
\end{equation*}
Fig.~\ref{fig1} shows that the discrete energy is monotonically
dissipated for all tested time steps, including the large time step
\(\Delta t=1\). 
Fig.\ \ref{fig2} presents the evolution of the $W^{1,\infty}$ semi-norm for different mesh sizes. The results indicate that $\norm{\grad\m_h}_{L^\infty}$ can approach $2\sqrt{2}h^{-1}$, suggesting possible finite-time blow-up. Meanwhile, the moment at which the maximum is attained coincides with the rapid energy decrease. These results are consistent with \cite{MR1974176, MR2318794}.
\begin{figure}
  \centering
  \includegraphics[scale = 0.2]{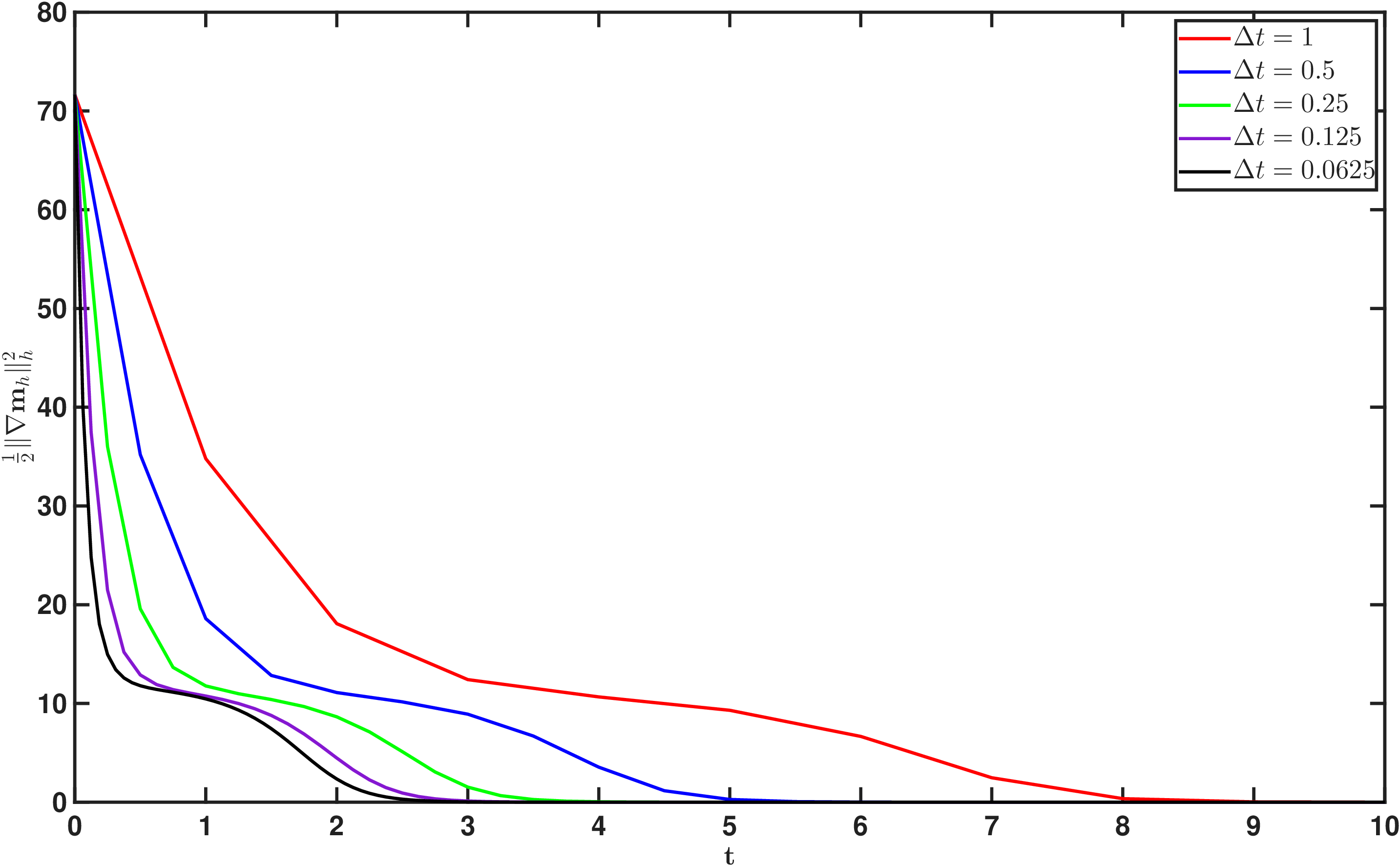}
  \caption{The evolution of the discrete energy for different $\Delta t$}\label{fig1}
\end{figure}
\begin{figure}
  \centering
  \includegraphics[scale = 0.2]{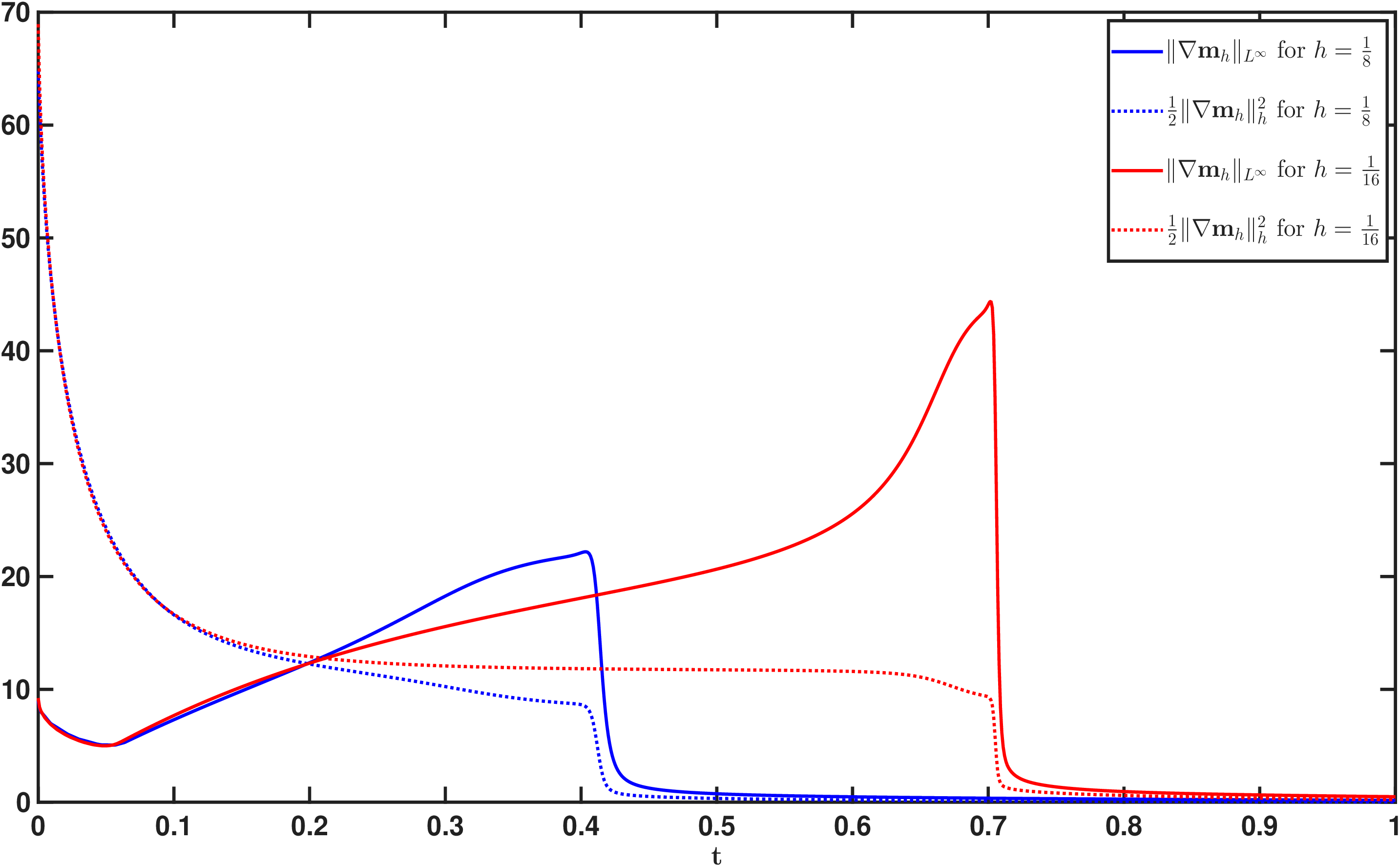}
  \caption{The evolution of $\|\grad \m_h\|_{L^\infty}$ and $\frac{1}{2}\norm{\grad \m_h}_h^2$ for different $h$.}\label{fig2}
\end{figure}
\section{Conclusion}\label{sec6}
In this work, we proposed and analyzed a linearly implicit mass-lumped
finite element projection method for the heat flow of harmonic maps. The
scheme preserves the unit-length constraint at all finite element nodes
and satisfies an unconditional discrete energy decay law. By combining
quadrature consistency estimates, edge-based cancellation identities, and
a bootstrap induction argument, we established unconditional optimal-order
error estimates without imposing any coupling condition between the time
step and the mesh size. We are currently preparing the extension to the LLG equation and higher-order schemes.
\appendix
\section{Proof of Lemma \ref{y9}}\label{t3}
Since $I_h\F_h|_K\in Q_1(K)$, by \eqref{t2} we have
\[
  \partial_i (I_h\F_h|_K)(\z_{K,i,\alpha'}^-)
  =
  \partial_i (I_h\F_h|_K)(\z_{K,i,\alpha'}^+)
  =
  \mathfrak D_{K,i}^{\alpha'}\F_h.
\]
Moreover,
\[
  \partial_i\F_h(\z_{K,i,\alpha'}^\pm)
  =
  (\mathfrak D_{K,i}^{\alpha'}\m_h^n)(\m_h^n\cdot\v_h)(\z_{K,i,\alpha'}^\pm)
  +
  \m_h^n(\z_{K,i,\alpha'}^\pm)\partial_i (\m_h^n\cdot\v_h)(\z_{K,i,\alpha'}^\pm).
\]
Therefore,
\begin{align}
&\partial_i\m_h^n(\z_{K,i,\alpha'}^+)\cdot
\l\partial_i\F_h(\z_{K,i,\alpha'}^+)-\partial_i I_h\F_h|_K(\z_{K,i,\alpha'}^+)\r\notag\\
&+\partial_i\m_h^n(\z_{K,i,\alpha'}^-)\cdot
\l\partial_i\F_h(\z_{K,i,\alpha'}^-)-\partial_i I_h\F_h|_K(\z_{K,i,\alpha'}^-)
\r\notag\\
=&
|\mathfrak D_{K,i}^{\alpha'}\m_h^n|^2\l(\m_h^n\cdot\v_h)(\z_{K,i,\alpha'}^+)+(\m_h^n\cdot\v_h)(\z_{K,i,\alpha'}^-)\r-2\mathfrak D_{K,i}^{\alpha'}\m_h^n\cdot \mathfrak D_{K,i}^{\alpha'}\F_h \notag\\
&+ \mathfrak D_{K,i}^{\alpha'}\m_h^n\cdot\l\m_h^n(\z_{K,i,\alpha'}^+)\partial_i (\m_h^n\cdot\v_h)(\z_{K,i,\alpha'}^+)+\m_h^n(\z_{K,i,\alpha'}^-)\partial_i (\m_h^n\cdot\v_h)(\z_{K,i,\alpha'}^-)\r. \label{identity-step-1}
\end{align}
Here, noting that $|\m_h^n(\z_{K,i,\alpha'}^\pm)|=1$, we have
\begin{align*}
  \m_h^n(\z_{K,i,\alpha'}^\pm)\cdot \mathfrak D_{K,i}^{\alpha'}\m_h^n
  =
  \pm\frac{h_{K,i}}{2}| \mathfrak D_{K,i}^{\alpha'}\m_h^n|^2.
\end{align*}
Using the above identity, we derive
\begin{align*}
  &\mathfrak D_{K,i}^{\alpha'}\m_h^n\cdot \mathfrak D_{K,i}^{\alpha'}\F_h\\
  =&
  \mathfrak D_{K,i}^{\alpha'}\m_h^n\cdot
  \frac{\m_h^n(\z_{K,i,\alpha'}^+) (\m_h^n\cdot\v_h)(\z_{K,i,\alpha'}^+)-\m_h^n(\z_{K,i,\alpha'}^-) (\m_h^n\cdot\v_h)(\z_{K,i,\alpha'}^-)}{h_{K,i}}\\
  =&
  \frac12 |\mathfrak D_{K,i}^{\alpha'}\m_h^n|^2
  \l(\m_h^n\cdot\v_h)(\z_{K,i,\alpha'}^+)+(\m_h^n\cdot\v_h)(\z_{K,i,\alpha'}^- )\r.
\end{align*}
Hence the first two terms on the right-hand side of
\eqref{identity-step-1} cancel and we get
\begin{align*}
&\partial_i\m_h^n(\z_{K,i,\alpha'}^+)\cdot
\l\partial_i\F_h(\z_{K,i,\alpha'}^+)-\partial_i I_h\F_h(\z_{K,i,\alpha'}^+)\r\notag\\
&+\partial_i\m_h^n(\z_{K,i,\alpha'}^-)\cdot
\l\partial_i\F_h(\z_{K,i,\alpha'}^-)-\partial_i I_h\F_h(\z_{K,i,\alpha'}^-)
\r\notag\\
=&\mathfrak D_{K,i}^{\alpha'}\m_h^n\cdot\l\m_h^n(\z_{K,i,\alpha'}^-)\partial_i (\m_h^n\cdot\v_h)(\z_{K,i,\alpha'}^-)+\m_h^n(\z_{K,i,\alpha'}^+)\partial_i (\m_h^n\cdot\v_h)(\z_{K,i,\alpha'}^+)\r.
\end{align*}
Multiplying by the quadrature weight $\omega_K$, summing over all
$i$-edges of $K$, all directions $i$, and all elements $K$, gives \eqref{t9}.

\section{Proof of Lemma \ref{lemma11}}\label{u6} We first note that
\begin{align*}
&\l \grad I_h\m(t_{n+1}) , \grad \v_h\r_h
-\l \grad \m(t_{n+1}), \grad \v_h \r \\
=&\l \grad I_h\m(t_{n+1}), \grad \v_h \r_h
-\l \grad I_h\m(t_{n+1}), \grad \v_h \r
+\l \grad (I_h\m(t_{n+1})-\m(t_{n+1})),\grad \v_h\r .
\end{align*}
The last term is estimated by the interpolation superconvergence result
\cite[Lemma 3.7]{MR4377027}:
\begin{equation*}
  \left|
  \l \grad (I_h\m(t_{n+1})-\m(t_{n+1})),\grad \v_h\r
  \right|
  \le Ch^2\norm{\v_h}_{H^1}.
\end{equation*}
It remains to prove
\begin{equation}\label{r1}
\left|
\l \grad I_h\m(t_{n+1}), \grad \v_h \r_h
-
\l \grad I_h\m(t_{n+1}), \grad \v_h \r
\right|
\le
Ch^2\norm{\m(t_{n+1})}_{H^3}\norm{\grad\v_h}.
\end{equation}

The case \(d=1\) is immediate since
\(\partial_1 I_h\m(t_{n+1})\) and \(\partial_1\v_h\) are constant on
each element and the two-point Gauss--Lobatto rule is exact. We therefore
consider \(d=2,3\). In the rest of the proof, for brevity, we write $\m:=\m(t_{n+1})$ and omit the restriction symbol on $K$. Since
\(I_h\m,\v_h\in Q_1(K)^3\), the tensor-product Gauss--Lobatto rule gives
\begin{align}\label{local-stiff-quad-id}
&(\nabla I_h\m,\nabla \v_h)_{K,h}
-(\nabla I_h\m,\nabla \v_h)_K \notag\\
=&
\sum_{i=1}^d
\sum_{\substack{j=1\\ j\ne i}}^d
\frac{h_{K,j}^2}{6}
\int_K
\partial_{ij}I_h\m\cdot\partial_{ij}\v_h\,\d\x
+
R_K(I_h\m,\v_h),
\end{align}
where \(R_K=0\) if \(d=2\). If \(d=3\), then
\begin{equation}\label{local-stiff-quad-rem}
  R_K(I_h\m,\v_h)
  =
  \sum_{i=1}^3
  \frac{h_{K,j_i}^2h_{K,\ell_i}^2}{36}
  \int_K
  \partial_{123}I_h\m\cdot\partial_{123}\v_h\,\d\x,
  \quad
  \{j_i,\ell_i\}=\{1,2,3\}\setminus\{i\}.
\end{equation}

We briefly justify \eqref{local-stiff-quad-id} and
\eqref{local-stiff-quad-rem}. For an interval \(I\), the identity
\[
  (p,q)_{I,h}-(p,q)_I
  =
  \frac{h_I^2}{6}\int_I p'q'\,\d x,
  \quad p,q\in P_1(I),
\]
is applied in the tensor-product directions. For fixed \(i\), the
integrand
\[
  \partial_i I_h\m\cdot\partial_i\v_h
\]
is independent of \(x_i\). In two dimensions there is only one transverse
direction, and this gives the corresponding term in
\eqref{local-stiff-quad-id}. In three dimensions, let
\(\{j_i,\ell_i\}=\{1,2,3\}\setminus\{i\}\), and let \(Q_r\) and \(I_r\)
denote the one-dimensional Gauss--Lobatto quadrature and the exact
integration in the \(x_r\)-direction. Then the quadrature error in the
two transverse directions is decomposed as
\[
  Q_{j_i}Q_{\ell_i}-I_{j_i}I_{\ell_i}
  =
  (Q_{j_i}-I_{j_i})I_{\ell_i}
  +
  I_{j_i}(Q_{\ell_i}-I_{\ell_i})
  +
  (Q_{j_i}-I_{j_i})(Q_{\ell_i}-I_{\ell_i}).
\]
The first two terms produce the \(h_{K,j_i}^2\)- and
\(h_{K,\ell_i}^2\)-contributions in \eqref{local-stiff-quad-id}, while
the last product term gives the corresponding contribution in
\eqref{local-stiff-quad-rem}.

Summing \eqref{local-stiff-quad-id} over all elements and using $
  \partial_{ij}I_h\m
  =
  \partial_{ij}(I_h\m-\m)+\partial_{ij}\m$,
we obtain
\begin{align*}
&\l \grad I_h\m,\grad\v_h\r_h
-
\l \grad I_h\m,\grad\v_h\r  \\
=&
\sum_{i=1}^d
\sum_{\substack{j=1\\j\ne i}}^d
\sum_{K\in\mathscr T_h}
\frac{h_{K,j}^2}{6}
\int_K
\partial_{ij}(I_h\m-\m)\cdot\partial_{ij}\v_h\,\d\x \\
&+
\sum_{i=1}^d
\sum_{\substack{j=1\\j\ne i}}^d
\sum_{K\in\mathscr T_h}
\frac{h_{K,j}^2}{6}
\int_K
\partial_{ij}\m\cdot\partial_{ij}\v_h\,\d\x
+
\sum_{K\in\mathscr T_h}R_K(I_h\m,\v_h) \\
\eqqcolon& R_1+R_2+R_3 .
\end{align*}

We estimate \(R_1,R_2,R_3\) separately. By the interpolation estimate and
the inverse inequality,
\[
  \|\partial_{ij}(I_h\m-\m)\|_{L^2(K)}
  \le Ch_K\|\m\|_{H^3(K)},
  \qquad
  \|\partial_{ij}\v_h\|_{L^2(K)}
  \le Ch_K^{-1}\|\nabla\v_h\|_{L^2(K)}.
\]
Therefore, by Cauchy's inequality,
\begin{equation*}
  |R_1|
  \le
  Ch^2\|\m\|_{H^3}\|\nabla\v_h\|.
\end{equation*}

We next estimate \(R_2\). For fixed \(i\ne j\), integrating by parts
elementwise in the \(x_i\)-direction gives
\[
\begin{aligned}
&\sum_{K\in\mathscr{T}_h}
\frac{h_{K,j}^2}{6}
\int_K
\partial_{ij}\m\cdot\partial_{ij}\v_h\,\d\x  \\
=&
-\sum_{K\in\mathscr{T}_h}
\frac{h_{K,j}^2}{6}
\int_K
\partial_{iij}\m\cdot\partial_j\v_h\,\d\x
+
\sum_{K\in\mathscr{T}_h}
\frac{h_{K,j}^2}{6}
\int_{\partial K}
n_i\partial_{ij}\m\cdot\partial_j\v_h\,\d S .
\end{aligned}
\]
We claim that the boundary terms vanish after summation over all elements. Set
\[
  B_{ij}
  :=
  \sum_{K\in\mathscr T_h}
  \frac{h_{K,j}^2}{6}
  \int_{\partial K}
  n_i\,\partial_{ij}\m\cdot\partial_j\v_h\,dS .
\]
Only faces orthogonal to \(x_i\) contribute to \(B_{ij}\), since \(n_i=0\)
on all other faces. Let \(F=K^-\cap K^+\) be an interior face orthogonal
to \(x_i\). The outward normals of \(K^-\) and \(K^+\) have opposite
\(i\)-components. Moreover, since the mesh is generated by tensor
products of one-dimensional partitions,
\[
  h_{K^-,j}=h_{K^+,j}.
\]
Because \(j\ne i\), \(\partial_j\v_h\) is the tangential derivative of
the single-valued trace of the continuous finite element function
\(\v_h\) on \(F\). Hence
\[
  \partial_j(\v_h|_{K^-})|_F
  =
  \partial_j(\v_h|_{K^+})|_F .
\]
Together with the single-valued trace of the smooth function
\(\partial_{ij}\m\), the two contributions from \(K^-\) and \(K^+\)
cancel. On an exterior face \(F\subset\partial\Omega\) orthogonal to
\(x_i\), the homogeneous Neumann condition gives
\(\partial_i\m=0\). Since \(j\ne i\), \(\partial_j\) is tangential to
\(F\), and therefore
\[
  \partial_{ij}\m
  =
  \partial_j(\partial_i\m)
  =
  0
  \quad\text{on }F
\]
in the tangential trace sense. Thus the exterior face contribution also
vanishes, and hence \(B_{ij}=0\).

Consequently,
\[
\begin{aligned}
  |R_2|
  &\le
  C h^2
  \sum_{i=1}^d
  \sum_{\substack{j=1\\j\ne i}}^d
  \sum_{K\in\mathscr T_h}
  \|\partial_{iij}\m\|_{L^2(K)}
  \|\partial_j\v_h\|_{L^2(K)}\le
  Ch^2\|\m\|_{H^3}\|\nabla\v_h\|.
\end{aligned}
\]

Finally, \(R_3=0\) if \(d=2\). If \(d=3\), the tensor-product
interpolation stability and the inverse inequality imply
\[
  \|\partial_{123}I_h\m\|_{L^2(K)}
  \le
  C\|\m\|_{H^3(K)},
  \qquad
  \|\partial_{123}\v_h\|_{L^2(K)}
  \le
  Ch_K^{-2}\|\nabla\v_h\|_{L^2(K)}.
\]
Therefore,
\[
\begin{aligned}
  |R_3|
  &\le
  Ch^4
  \sum_{K\in\mathscr T_h}
  \|\partial_{123}I_h\m\|_{L^2(K)}
  \|\partial_{123}\v_h\|_{L^2(K)}\le
  Ch^2\|\m\|_{H^3}\|\nabla\v_h\|.
\end{aligned}
\]

Combining the estimates for \(R_1,R_2,R_3\), we obtain \eqref{r1}.
Together with the preceding interpolation superconvergence estimate, this
yields
\[
  \left|
  \l \grad I_h\m(t_{n+1}),\grad\v_h\r_h
  -
  \l \grad\m(t_{n+1}),\grad\v_h\r
  \right|
  \le
  Ch^2\|\v_h\|_{H^1}.
\]
This completes the proof.
\section{Proof of Lemma \ref{lemma3}}\label{y5}
Recall that
\[
\l \tilde{\Delta}_h\v_h, \w_h \r_h
= - \l \grad \v_h, \grad \w_h \r_h,\quad
\l \Delta_h\v_h, \w_h \r
= - \l \grad \v_h, \grad \w_h \r,
\quad \forall \v_h,\w_h\in\S_h^1 .
\]
Since \(\partial_i\v_h\) and \(\partial_i\w_h\) are constant in the
\(x_i\)-direction, Lemma~\ref{lemma1} gives contributions only from the
directions \(j\ne i\). Using Cauchy's inequality, the local inverse
estimate
\[
  h_K^2\|\partial_i\partial_j\w_h\|_{L^2(K)}
  \le C\|\w_h\|_{L^2(K)},
\]
and the norm equivalence \eqref{b3}, we obtain, for any
\(\w_h\in\S_h^1\),
\begin{align}
  |(\Delta_h\v_h,\w_h)|
  &\le
  |(\tilde\Delta_h\v_h,\w_h)_h|
  +
  \left|
  \l\grad\v_h,\grad\w_h\r_h
  -
  \l\grad\v_h,\grad\w_h\r
  \right|  \notag\\
  &\le
  \norm{\tilde\Delta_h\v_h}_h\norm{\w_h}_h 
  +C
  \left(
  \sum_{K\in\mathscr T_h}
  \sum_{i=1}^d\sum_{j\ne i}
  \norm{\partial_i\partial_j\v_h}_{L^2(K)}^2
  \right)^{1/2}\norm{\w_h}_h .
  \label{aux-p2-1}
\end{align}

It remains to prove
\begin{equation}\label{aux-p2-2}
  \sum_{K\in\mathscr T_h}
  \sum_{i=1}^d\sum_{j\ne i}
  \norm{\partial_i\partial_j\v_h}_{L^2(K)}^2
  \le
  C\norm{\tilde\Delta_h\v_h}_h^2 .
\end{equation}
For \(i=1,\ldots,d\), define the directional discrete Laplacian
\(\delta_{h,i}^2:\S_h^1\to\S_h^1\) by
\[
  \l \delta_{h,i}^2\v_h,\w_h\r_h
  =
  -\l \partial_i\v_h,\partial_i\w_h\r_h,
  \qquad \forall \w_h\in\S_h^1 .
\]
Then
\[
  \tilde\Delta_h\v_h=\sum_{i=1}^d\delta_{h,i}^2\v_h .
\]

We justify the mixed-derivative identity by the tensor-product matrix
representation. Let \(M_\ell\) and \(S_\ell\) denote the one-dimensional
lumped mass matrix and stiffness matrix in the \(x_\ell\)-direction.
Then
\[
  M_h=M_1\otimes\cdots\otimes M_d,
\]
and the directional stiffness matrix corresponding to the \(x_i\)-direction is
\[
  A_i
  =
  M_1\otimes\cdots\otimes M_{i-1}
  \otimes S_i
  \otimes M_{i+1}\otimes\cdots\otimes M_d .
\]
If \(\mathbf v\) is the nodal coefficient vector of \(\v_h\), then
\(\delta_{h,i}^2\v_h\) has coefficient vector
\(-M_h^{-1}A_i\mathbf v\). Hence, for \(i\ne j\),
\[
  \l\delta_{h,i}^2\v_h,\delta_{h,j}^2\v_h\r_h
  =
  \mathbf v^\top A_iM_h^{-1}A_j\mathbf v .
\]
Since \(M_h^{-1}=M_1^{-1}\otimes\cdots\otimes M_d^{-1}\), a direct
Kronecker-product calculation gives
\[
  A_iM_h^{-1}A_j
  =
  S_i\otimes S_j
  \otimes
  \bigotimes_{\ell\ne i,j}M_\ell ,
  \qquad i\ne j,
\]
up to the natural ordering of tensor factors. The matrix on the
right-hand side is precisely the tensor-product Gauss--Lobatto matrix
for the broken norm of the mixed derivative \(\partial_i\partial_j\v_h\).
Consequently,
\begin{equation*}
  \l\delta_{h,i}^2\v_h,\delta_{h,j}^2\v_h\r_h
  =
  \|\partial_i\partial_j\v_h\|_h^2,
  \qquad i\ne j .
\end{equation*}

Now,
\begin{align}
  \norm{\tilde\Delta_h\v_h}_h^2
  &=
  \norm{\sum_{i=1}^d\delta_{h,i}^2\v_h}_h^2 =
  \sum_{i=1}^d\norm{\delta_{h,i}^2\v_h}_h^2
  +
  2\sum_{1\le i<j\le d}
  \l\delta_{h,i}^2\v_h,\delta_{h,j}^2\v_h\r_h \notag\\
  &=
  \sum_{i=1}^d\norm{\delta_{h,i}^2\v_h}_h^2
  +
  2\sum_{1\le i<j\le d}
  \norm{\partial_i\partial_j\v_h}_{h}^2 \ge
  \sum_{i=1}^d\sum_{j\ne i}
  \norm{\partial_i\partial_j\v_h}_{h}^2 .
  \label{o7}
\end{align}
By using the equivalence \eqref{b3} on each element, \eqref{o7} implies \eqref{aux-p2-2}.

Finally, taking the supremum over \(\w_h\in\S_h^1\), \(\w_h\ne0\) in
\eqref{aux-p2-1} and using \(\|\w_h\|\simeq\|\w_h\|_h\), we obtain
\[
  \|\Delta_h\v_h\|
  \le
  C\left[
  \|\tilde\Delta_h\v_h\|_h
  +
  \left(
  \sum_{K\in\mathscr T_h}
  \sum_{i=1}^d\sum_{j\ne i}
  \|\partial_i\partial_j\v_h\|_{L^2(K)}^2
  \right)^{1/2}
  \right]
  \le
  C_5\|\tilde\Delta_h\v_h\|_h .
\]
This proves \eqref{p2}.

\section{Proof of Lemma \ref{lemma10}}\label{u3}
We prove the estimate elementwise. For simplicity, we denote
\[
  \z^-:=\z_{K,i,\alpha'}^-,
  \qquad
  \z^+:=\z_{K,i,\alpha'}^+,
  \qquad
  h_i:=h_{K,i},\qquad \mathfrak D\coloneqq\mathfrak D_{K,i}^{\alpha'}
\]
Along the local edge \((\z^-,\z^+)\), we compute
\begin{align*}
  \mathfrak D\e_h^n =& \mathfrak D(I_h\m(t_n) - \m_h^n)\\
  =&\mathfrak D I_h\m(t_n) -\frac{1}{h_{i}}\l \frac {\widetilde\m_h^n(\z^+)}{|\widetilde\m_h^n(\z^+)|} - \frac {\widetilde\m_h^n(\z^+)}{|\widetilde\m_h^n(\z^-)|} + \frac {\widetilde\m_h^n(\z^+)}{|\widetilde\m_h^n(\z^-)|} - \frac {\widetilde\m_h^n(\z^-)}{|\widetilde\m_h^n(\z^-)|}\r\\
  =&\mathfrak D I_h\m(t_n) -\widetilde\m_h^n(\z^+) \mathfrak D\frac {1}{|\widetilde\m_h^n|} - \frac {\mathfrak D \widetilde\m_h^n}{|\widetilde\m_h^n(\z^-)|}  \\
  =&\l1-\frac1{|\widetilde\m_h^n(\z^-)|}\r \mathfrak D I_h\m(t_n) -\widetilde\m_h^n(\z^+) \mathfrak D \frac {1}{|\widetilde\m_h^n|} + \frac {\mathfrak D \widetilde\e_h^n}{|\tilde\m_h^n(\z^-)|}.
\end{align*}
Using that $|\tilde\m_h^n(\z^-)|\ge 1$, we have
\[
\left|1-\frac1{|\widetilde\m_h^n(\z^-)|}\right| = \left|\frac{|\widetilde\m_h^n(\z^-)|-|I_h\m(t_n)(\z^-)|}{|\widetilde\m_h^n(\z^-)|}\right|\le|\widetilde\e_h^n(\z^-)|.
\]
Moreover, by \eqref{b1} and the nodal identity
\(|I_h\m(t_n)(\z)|=1\), for sufficiently small \(h\) and
\(\Delta t\),
\[
  |\widetilde\m_h^n(\z)|
  \le |I_h\m(t_n)(\z)|+|\widetilde\e_h^n(\z)|
  \le 2,
  \qquad \forall \z\in\mathcal Q.
\]
Together with 
\begin{equation}\label{p9}
\norm{\nabla I_h\m(t_n)}_{L_h^\infty}\le C\norm{\nabla I_h\m(t_n)}_{L^\infty}\le C\norm{\m(t_{n})}_{W^{1,\infty}},
\end{equation}
we have for sufficiently small $h$ and $\Delta t$
\begin{align*}
  |\mathfrak D \e_h^n|\le C|\widetilde\e_h^n(\z^-)| + C|\mathfrak D \widetilde\e_h^n| + C\left|\mathfrak D \frac1{|\widetilde\m_h^n|}\right|
\end{align*}
We estimate the last term, decomposing \(\widetilde\m_h^n\) as $I_h\m(t_n)-\widetilde\e_h^n$, we obtain
\begin{align*}
&\left|
|\widetilde\m_h^n(\z^+)|^2 - |\widetilde\m_h^n(\z^-)|^2 \right| \\
=&\left| -2 I_h\m(t_n)(\z^+)\cdot\widetilde\e_h^n(\z^+) + 2 I_h\m(t_n)(\z^-)\cdot\widetilde\e_h^n(\z^-) + |\widetilde\e_h^n(\z^+)|^2 - |\widetilde\e_h^n(\z^-)|^2 \right| \\
\le&C h_i\l |I_h\m(t_n)(\z^+)| |\mathfrak D\widetilde\e_h^n| +  |\mathfrak D I_h\m(t_n)| |\widetilde\e_h^n(\z^-)| + |\mathfrak D\widetilde\e_h^n|( |\widetilde\e_h^n(\z^-)| + |\widetilde\e_h^n(\z^+)| )\r.
\end{align*}
Using \eqref{b1} and \eqref{p9} and noting that $|a - b| = \frac{|a^2 - b^2|}{a+b}\le |a^2- b^2|$ for $a, b\ge 1$, we obtain
\begin{equation*}
\frac{\left| |\widetilde\m_h^n(\z^+)| - |\widetilde\m_h^n(\z^-)| \right| }{h_i}\le C|\mathfrak D\widetilde\e_h^n| + C
|\widetilde\e_h^n(\z^-)| .
\end{equation*}
Since the function $s\mapsto 1/s$
is Lipschitz on \([1,\infty)\), the above inequality implies
\begin{equation*}
\left| \mathfrak D \frac1{|\widetilde\m_h^n|} 
\right| \le C|\mathfrak D\widetilde\e_h^n| +C
|\widetilde\e_h^n(\z^-)|.
\end{equation*}

Thus, we have the local edge estimate
\begin{equation*}
|\mathfrak D\e_h^n| \le C|\mathfrak D\widetilde\e_h^n| + C\left(
|\widetilde\e_h^n(\z^-)|+ |\widetilde\e_h^n(\z^+)| \right),
\end{equation*}
which further gives
\begin{align*}
  \norm{\partial_i\e_h^n}_{L_h^2(K)}^2
  &\le C\norm{\partial_i\widetilde\e_h^n}_{L_h^2(K)}^2 + C \sum_{\alpha'\in\{0,1\}^{d-1}}
  \omega_K \left( |\widetilde\e_h^n(\z_{K,i,\alpha'}^-)|^2 + |\widetilde\e_h^n(\z_{K,i,\alpha'}^+)|^2\right) \\
  &\le C\norm{\partial_i\widetilde\e_h^n}_{L_h^2(K)}^2 + C\norm{\widetilde\e_h^n}_{L_h^2(K)}^2.
\end{align*}
Summing over \(i=1,\ldots,d\) and over all \(K\in\mathscr T_h\), we finish the proof.
\section{Proof of Lemma \ref{lemma8}}\label{u4}
Assume that the grid points in the \(x_i\)-direction are
\(\{x_{i,k}\}_{k=0}^{N_i}\). Let
\[
  \z=\z_{\mathbf k}
  =(x_{1,k_1},\ldots,x_{d,k_d})
\]
be an interior node. For \(1\le k_i\le N_i-1\), set
\[
  h_i^+=x_{i,k_i+1}-x_{i,k_i},
  \qquad
  h_i^-=x_{i,k_i}-x_{i,k_i-1}.
\]
and define
\[
  \z_{i,+}\coloneqq\z_{\mathbf k+\mathbf e_i},
  \qquad
  \z_{i,-}\coloneqq\z_{\mathbf k-\mathbf e_i}.
\]
For a grid function \(\f\), define
\[
  D_i^+\f(\z)
  =
  \frac{\f(\z_{i,+})-\f(\z)}{h_i^+},
  \qquad
  D_i^-\f(\z)
  =
  \frac{\f(\z)-\f(\z_{i,-})}{h_i^-}.
\]

On a Cartesian tensor-product \(Q_1\) mass-lumped grid, the nodal values
of \(\widetilde\Delta_h\v_h\) coincide at interior nodes with those of
the finite-difference operator
\[
  \Delta_h^{\rm FD} \v_h(\z)
  =
  \sum_{i=1}^d
  \frac{2}{h_i^++h_i^-}
  \left(D_i^+\v_h(\z)-D_i^-\v_h(\z)\right).
\]
At boundary nodes, the corresponding one-sided stencil is obtained by
omitting the nonexistent forward or backward differences. More precisely,
if \(k_i=0\), only \(D_i^+\) appears in the \(i\)-th direction, while if
\(k_i=N_i\), only \(D_i^-\) appears. The coefficients in these one-sided
stencils satisfy the same local boundedness property as those in the
interior stencil. Hence the pointwise estimate derived below remains
valid at boundary nodes, with the nonexistent differences omitted and
with a possibly different constant independent of \(h\).

We next show that the derivatives of the nodal projection are uniformly
bounded on the relevant line segments. By \eqref{b1} and the regularity
of the exact solution, for sufficiently small \(h\) and \(\Delta t\),
\[
\begin{aligned}
  |\widetilde\m_h^n(\z_{i,\pm})-\widetilde\m_h^n(\z)|
  &\le
  |I_h\m(t_n)(\z_{i,\pm})-I_h\m(t_n)(\z)|
  +|\widetilde\e_h^n(\z_{i,\pm})|
  +|\widetilde\e_h^n(\z)|  \\
  &\le
  Ch+2\|\widetilde\e_h^n\|_{L_h^\infty}
  \le \frac12 .
\end{aligned}
\]
Since \(|\widetilde\m_h^n(\z)|\ge1\) at all nodes, every point
\[
  \bm\xi_{i,\pm}
  =
  \widetilde\m_h^n(\z)
  +\theta\bigl(\widetilde\m_h^n(\z_{i,\pm})
  -\widetilde\m_h^n(\z)\bigr),
  \qquad 0\le\theta\le1,
\]
on the segment joining \(\widetilde\m_h^n(\z)\) and
\(\widetilde\m_h^n(\z_{i,\pm})\) satisfies
\[
  |\bm\xi_{i,\pm}|
  \ge
  |\widetilde\m_h^n(\z)|
  -
  |\widetilde\m_h^n(\z_{i,\pm})-\widetilde\m_h^n(\z)|
  \ge \frac12 .
\]

Let
\[
  P(\s)=\frac{\s}{|\s|},
  \qquad
  \s\in\mathbb R^3\setminus\{\mathbf0\}.
\]
Here \(D\) and \(D^2\) denote the first and second Fréchet derivatives
with respect to \(\s\). We use \(\|DP(\s)\|\) and
\(\|D^2P(\s)\|\) for the operator norm of
the map \(DP(\s):\mathbb R^3\to\mathbb R^3\) and
\(D^2P(\s):\mathbb R^3\times\mathbb R^3\to\mathbb R^3\) respectively.
Since \(|\bm\xi_{i,\pm}|\ge1/2\), and since
\[
  \|DP(\s)\|\le C|\s|^{-1},
  \qquad
  \|D^2P(\s)\|\le C|\s|^{-2},
\]
we have
\[
  \|DP(\bm\xi_{i,\pm})\|
  +
  \|D^2P(\bm\xi_{i,\pm})\|
  \le C .
\]

Taylor's formula gives
\[
\begin{aligned}
  D_i^\pm P(\widetilde\m_h^n)(\z)
  &=
  DP(\widetilde\m_h^n(\z))D_i^\pm\widetilde\m_h^n(\z)
  \pm
  \frac{h_i^\pm}{2}
  D^2P(\bm\xi_{i,\pm})
  \left[
    D_i^\pm\widetilde\m_h^n(\z),
    D_i^\pm\widetilde\m_h^n(\z)
  \right].
\end{aligned}
\]
Therefore, at an interior node $\z$,
\[
\begin{aligned}
  &(\widetilde\Delta_h\m_h^n)(\z)=
  \Delta_h^{\rm FD}P(\widetilde\m_h^n)(\z)  \\
  &=
  DP(\widetilde\m_h^n(\z))
  (\widetilde\Delta_h\widetilde\m_h^n)(\z)+
  \sum_{i=1}^d
  \frac{h_i^+}{h_i^++h_i^-}
  D^2P(\bm\xi_{i,+})
  \left[
    D_i^+\widetilde\m_h^n(\z),
    D_i^+\widetilde\m_h^n(\z)
  \right]  \\
  &\quad+
  \sum_{i=1}^d
  \frac{h_i^-}{h_i^++h_i^-}
  D^2P(\bm\xi_{i,-})
  \left[
    D_i^-\widetilde\m_h^n(\z),
    D_i^-\widetilde\m_h^n(\z)
  \right].
\end{aligned}
\]
The same bound holds at boundary nodes with the nonexistent one-sided
differences omitted. Consequently, for all $\z\in\mathcal{Q}$
\begin{equation}\label{p6-pointwise}
  |(\widetilde\Delta_h\m_h^n)(\z)|
  \le
  C|(\widetilde\Delta_h\widetilde\m_h^n)(\z)|
  +
  C\sum_{i=1}^d
  \left(
    |D_i^+\widetilde\m_h^n(\z)|^2
    +
    |D_i^-\widetilde\m_h^n(\z)|^2
  \right).
\end{equation}

Multiplying \eqref{p6-pointwise} by the nodal quadrature weights and
summing over all nodes yields
\begin{align}\label{t8}
  \|\widetilde\Delta_h\m_h^n\|_h^2
  &\le
  C\|\widetilde\Delta_h\widetilde\m_h^n\|_h^2
  +
  C\sum_{K\in\mathscr T_h}
  \sum_{\z\in\mathcal Q_K}
  \omega_K
  \sum_{i=1}^d
  |\partial_i(\widetilde\m_h^n|_K)(\z)|^4  \notag\\
  &\le
  C\|\widetilde\Delta_h\widetilde\m_h^n\|_h^2
  +
  C\|\nabla\widetilde\m_h^n\|_{L_h^4}^4 .
\end{align}
Here we used the edge-difference representation of the broken
Gauss--Lobatto gradient norm and the finite overlap of neighboring
one-sided differences.

It remains to bound the two terms on the right-hand side. By \eqref{u8}
and the induction assumption \eqref{e1},
\[
  \|\widetilde\Delta_h\widetilde\m_h^n\|_h
  \le
  \|\widetilde\Delta_h I_h\m(t_n)\|_h
  +
  \|\widetilde\Delta_h\widetilde\e_h^n\|_h
  \le C+M .
\]
Next, using \eqref{p7}, \eqref{b6}, Lemma~\ref{lemma3}, and
\eqref{e1}, we obtain
\[
\begin{aligned}
  \|\nabla\widetilde\e_h^n\|_{L_h^4}
  &\le
  C\|\nabla\widetilde\e_h^n\|_{L^4}
  \le
  C\|\nabla\widetilde\e_h^n\|_{L^6}\le
  C\|\Delta_h\widetilde\e_h^n\|
  \le
  C\|\widetilde\Delta_h\widetilde\e_h^n\|_h
  \le CM .
\end{aligned}
\]
Thus,
\[
  \|\nabla\widetilde\m_h^n\|_{L_h^4}
  \le
  \|\nabla I_h\m(t_n)\|_{L_h^4}
  +
  \|\nabla\widetilde\e_h^n\|_{L_h^4}
  \le C(1+M).
\]
Combining the above estimates in \eqref{t8}, we obtain
\[
  \|\widetilde\Delta_h\m_h^n\|_h
  \le C(1+M)^2\le C_0(M) \coloneqq\max\{1, C(1+M)^2\}.
\]
\section*{Acknowledgments}
\bibliographystyle{siamplain}
\bibliography{references}
\end{document}